\documentclass[11pt]{article}
\usepackage{amssymb,amsmath,amsthm,graphics,amscd,amsfonts}

\theoremstyle{plain}
\newtheorem{theorem}{Theorem}[section]
\newtheorem{proposition}[theorem]{Proposition}
\newtheorem{corollary}[theorem]{Corollary}
\newtheorem{lemma}[theorem]{Lemma}

\newtheorem{example}[theorem]{Example}

\newtheorem{remark}[theorem]{Remark}

\begin{document}

\title{
	Multiplier ideal sheaves and the K\"ahler-Ricci flow on 
	toric Fano manifolds with large symmetry
}

\author{Yuji Sano
\thanks{
		Department of Mathematics, 
		Kyushu University, 
		6-10-1, Hakozaki, Higashiku, Fukuoka-city, 
		Fukuoka 812-8581 Japan
\endgraf
{\it E-mail address\/}:
	sano@math.kyushu-u.ac.jp
\endgraf
\endgraf
}}

\date{December 10, 2008}

\maketitle

\begin{abstract} 
The purpose of this paper is to calculate the support of the multiplier ideal sheaves derived from the K\"ahler-Ricci flow on certain toric Fano manifolds with large symmetry.
The early idea of this paper has already been in Appendix of \cite{futaki-sano0711}.
\end{abstract}

\section{Introduction}
In \cite{futaki-sano0711}, Futaki and the author investigated the relationship between the multiplier ideal subvariety derived from the continuity method on toric Fano manifolds and Futaki invariant, and calculated the multiplier ideal subvariety on a simple example.
On the other hand, the relationship between the multiplier ideal sheaves and the K\"ahler-Ricci flow has recently been studied.
The first work on this topic is given by Phong-Sesum-Sturm \cite{pss0611}.
They give a sufficient and necessary condition for the convergence of the K\"ahler-Ricci flow in the terms of the multiplier ideal sheaves.
After \cite{pss0611} Rubinstein \cite{rubinstein0708} proves that the K\"ahler-Ricci flow will induce a multiplier ideal sheaf satisfying the same properties as Nadel's multiplier ideal sheaves derived from the continuity method.
The purpose of this paper is to calculate the multiplier ideal subvarieties from the K\"ahler-Ricci flow  in the sense of \cite{rubinstein0708}  on certain toric Fano manifolds with large symmetry.
Our method owes largely to the result about the convergence of the K\"ahler-Ricci flow proved by Tian-Zhu \cite{tian-zhu06}. 
More precisely, they proved that if $X$ admits a K\"ahler-Ricci soliton then the K\"ahler-Ricci flow will converge to it in the sense of Cheeger-Gromov, so we shall calculate the multiplier ideal subvarieties from the data of K\"ahler-Ricci solitons in the case of toric Fano manifolds with large symmetry.
The early idea of this paper has already been in Appendix of \cite{futaki-sano0711}.

First of all, let us recall about the K\"ahler-Ricci flow on Fano manifolds.
Let $(X, \omega)$ be a Fano manifold with a K\"ahler form $\omega$ representing $c_1(X)$.
The normalized K\"ahler-Ricci flow on $X$ is defined by
\begin{equation}\label{eq:ricci_flow}
	\frac{d}{dt}\omega_t=-\mbox{Ric}(\omega_t)+\omega_t
\end{equation}
where $t\in \mathbb{R}_{\ge 0}$, $\mbox{Ric}(\omega_t)$ is the Ricci form of $\omega_t$ and $\omega_0=\omega\in c_1(X)$.
Since 
the flow (\ref{eq:ricci_flow}) preserves the K\"ahler class,
we can consider the corresponding equation to (\ref{eq:ricci_flow}) with respect to K\"ahler potentials
\begin{eqnarray}
	\label{eq:ricci_flow_potential}
		\left\{\begin{array}{l}
			\frac{\partial\varphi_t}{\partial t}
			=
			\log \frac{\det (g_{i\bar{j}}+\varphi_{i\bar{j}})}{\det (g_{i\bar{j}})}+\varphi_t -h_0,
		\\
			\varphi_0 \equiv c_0			
		\end{array}\right.
\end{eqnarray}
where $\omega_t=\omega_0+\frac{\sqrt{-1}}{2\pi}\partial\bar{\partial}\varphi_t$, $c_0$ is a constant and $h_0$ is a real-valued function determined by
\begin{equation}\label{eq:ricci_discrepancy}
	\mbox{Ric}(\omega_0)-\omega_0=\frac{\sqrt{-1}}{2\pi}\partial\bar{\partial}h_0,
	\,\,\,
	\int_Xe^{h_0}\omega_0^n=\int_X\omega_0^n.
\end{equation}
The existence of the solution of (\ref{eq:ricci_flow_potential}) for all $t>0$ is proved by Cao \cite{cao85} by following Yau's argument in \cite{yau78}.
If the K\"ahler-Ricci flow converges in $C^\infty$-sense, the limit is a K\"ahler-Einstein metric.
However, since there are some obstructions for the existence of K\"ahler-Einstein metrics on Fano manifolds (\cite{matsushima57}, \cite{futaki83}, \cite{tian97}), the K\"ahler-Ricci flow does not necessarily converge on Fano manifolds.
On the other hand, it has been conjectured that the existence of canonical K\"ahler metrics including K\"ahler-Einstein metrics would be equivalent to certain stability of manifolds in the sense of Geometric Invariant Theory (cf. \cite{tian97}, \cite{donaldson0200}). 
This conjecture is an analogue of the Hitchin-Kobayashi correspondence between holomorphic Hermitian-Einstein vector bundles and slope polystable vector bundles.
So we expect that the convergence condition of the K\"ahler-Ricci flow would be described in terms of GIT.
To be more concrete, if $X$ does not admit K\"ahler-Einstein metrics, we expect that the K\"ahler-Ricci flow would induce an obstruction to the existence of K\"ahler-Einstein metrics corresponding to the destabilizing subsheaves in the Hitchin-Kobayashi correspondence for vector bundles. 
In the case of K\"ahler-Einstein metrics on Fano manifolds, a candidate for such obstruction sheaves is the so-called ``multiplier ideal sheaf" introduced by Nadel in \cite{nadel90}. 
Nadel proved that if $X$ does not admit K\"ahler-Einstein metrics, then the failure of the closedness condition for the continuity method induces a multiplier ideal sheaf.
(This fact can be extended in the cases of other canonical K\"ahler metrics such as K\"ahler-Ricci solitons \cite{futaki-sano0711} and K\"ahler-Einstein metrics in the sense of Mabuchi \cite{sano08}.)
The analogous result for the K\"ahler-Ricci flow was proved recently by Rubinstein \cite{rubinstein0708}. 
To explain the result of \cite{rubinstein0708}, first let us recall the definition of multiplier  ideal sheaves. In this paper, we adopt the formulation introduced by Demailly-Koll\'ar in \cite{demailly-kollar01}. 
Let $\psi$ be an almost plurisubharmonic function on $X$, i.e., $\psi$ is written locally as a sum of a plurisubharmonic function and a smooth function.
For $\psi$, we define a multiplier ideal sheaf $\mathcal{I}(\psi)\subset \mathcal{O}_X$ as follows; 
for every open subset $U\subset X$, the space $\Gamma(U, \mathcal{I}(\psi))$ of local sections of $\mathcal{I}(\psi)$ over $U$ is given by
\[
	\Gamma(U, \mathcal{I}(\psi))=
	\{
		f\in \mathcal{O}_X(U)
		\mid
		\int_U |f|^2 e^{-\psi}d\nu <\infty
	\},
\]
where $f$ is a holomorphic function on $U$ and $d\nu$ is a fixed volume form on $X$. 
Note that $\mathcal{I}(\psi)$ is a coherent ideal sheaf (cf.\cite{demailly-kollar01}) and invariant up to an additive constant.  
Multiplier ideal sheaves describe the singularities of almost plurisubharmonic functions. 
The result of  \cite{rubinstein0708} is as follows. 
\begin{theorem}[\cite{rubinstein0708}]\label{thm:PSS_Rubinstein}
Let $(X, \omega)$ be an $n$-dimensional Fano manifold with a K\"ahler form $\omega$ in $c_1(X)$, and $G\subset \mbox{Aut}(X)$ be a compact subgroup of the group $\mbox{Aut}(X)$ of holomorphic automorphisms of $X$. Let $\gamma \in (n/(n+1), 1)$. Suppose that $X$ does not admit K\"ahler-Einstein metrics.
Then there is an initial condition $c_0$ in (\ref{eq:ricci_flow_potential}) and a sequence $\{\varphi_{t_j}\}_{j\ge 0}$ such that $\varphi_{t_j}-\sup \varphi_{t_j}$ converges to an almost plurisubharmonic function $\varphi_{\infty}$ in $L^1$-topology and the associated multiplier ideal sheaf $\mathcal{I}(\gamma \varphi_{\infty})$ is $G^\mathbb{C}$-invariant and proper, i.e., $\mathcal{I}(\gamma \varphi_{\infty})$ equals neither to $0$ nor $\mathcal{O}_X$, where $G^\mathbb{C}$ is the complexification of $G$.
\end{theorem}
Remark that in \cite{rubinstein0708} the multiplier ideal sheaf is constructed from the sequence of $\{\varphi_t-\int_X\varphi_t\omega^n\}_t$ instead of $\{\varphi_{t}-\sup \varphi_{t}\}_t$ but there is no difference between them due to a standard argument by the Green function, more precisely, there is a constant $C$ such that
$
	\sup \varphi_t -C \le \int_X\varphi_t \omega^n \le \sup\varphi_t.
$
In order to get the limit in $L^1$-topology, we need to consider the family of the \textit{sifted} K\"ahler potentials $\{\varphi_t-\int_X \varphi_t\omega_0^n\}_t$ (equivalently $\{\varphi_t-\sup \varphi_t\}_t$).

Phong-Sesum-Sturm \cite{pss0611} (see also \cite{phong-sturm0801}) prove that if the K\"ahler-Ricci flow does not converge, then the \textit{non-shifted} solution $\{\varphi_t\}_t$ of (\ref{eq:ricci_flow_potential}) with respect to an appropriate initial condition will induce another proper multiplier ideal sheaf $\mathcal{J}^\gamma$ for $\gamma >1$, which is defined as follows; 
for every open subset $U\subset X$, the space $\Gamma(U, \mathcal{J}^\gamma)$ of local sections of $\mathcal{J}^\gamma$ over $U$ is given by
\[
	\Gamma(U, \mathcal{J}^\gamma):=
	\{
		f\in\mathcal{O}_X(U) \mid \sup_{t\ge 0}\int_U |f|^2e^{-\gamma\varphi_t}
		d\nu<\infty
	\}.	
\]
Furthermore they give a necessary and sufficient condition for the convergence of the K\"ahler-Ricci flow in terms of $\mathcal{J}^\gamma$, more precisely, the K\"ahler-Ricci flow converges if and only if there exists $\gamma>1$ such that $\mathcal{J}^\gamma$ admit the global section $1$.

The difference between the Nadel's type multiplier ideal sheaves $\mathcal{I}(\gamma\varphi_\infty)$ in Theorem \ref{thm:PSS_Rubinstein} and $\mathcal{J}^\gamma$ in \cite{phong-sturm0801} appears in the following vanishing theorem which is one of the important properties of the multiplier ideal sheaves;
\begin{theorem}[Nadel's vanishing theorem, \cite{nadel90, demailly-kollar01}]\label{thm:vanishing_thm}
Let $(X, \omega)$ be a compact K\"ahler manifold and $L$ be a holomorphic line bundle over $X$ with a singular Hermitian metric $h=e^{-\psi}h_0$, where $h_0$ is a smooth Hermitian metric and $\psi$ is an almost plurisubharmonic function. Suppose that the curvature form $\Theta(h)=-\frac{\sqrt{-1}}{2\pi}\partial\bar{\partial}\log h$ is positive definite in the current sense, that is to say, $\Theta(h)\ge \epsilon \omega$ for some $\epsilon >0$. Then we have
\begin{equation}\label{eq:nadel_vanishing}
	H^q(X, K_X\otimes L \otimes \mathcal{I}(\psi))=0, \,\,\, q>0,
\end{equation}
where $K_X$ is the canonical bundle.
\end{theorem}
Applying the above theorem to $L=K^{-1}_X$ and $\psi=\gamma \varphi_\infty$ for $\gamma \in (n/(n+1), 1)$ in Theorem \ref{thm:PSS_Rubinstein}, we find that 
\begin{equation}\label{eq:vanishing_theorem2}
	\begin{array}{l}
		H^0(V_\gamma, \mathcal{O}_{V_\gamma})=\mathbb{C}, 
	\,\,\,
		H^q(V_\gamma, \mathcal{O}_{V_\gamma})=0
	\end{array}
\end{equation}
for all $q>0$, where $V_\gamma$ is the associated subscheme of $\mathcal{I}(\gamma \varphi_\infty)$ whose structure sheaf $\mathcal{O}_{V_\gamma}=\mathcal{O}_X/\mathcal{I}(\gamma \varphi_\infty)$.
(\ref{eq:vanishing_theorem2}) gives us some geometric properties of $V_\gamma$ such as the connectedness, etc.
See \cite{nadel90, demailly-kollar01} for the other properties of $V_\gamma$. 
Remark that the multiplier ideal sheaf in \cite{phong-sturm0801} does not need to satisfy (\ref{eq:vanishing_theorem2}), because $\gamma>1$.
In this paper, we call $V_\gamma$ derived in Theorem \ref{thm:PSS_Rubinstein} the \textbf{KRF-multiplier ideal subscheme} (KRF-MIS) of 
exponent $\gamma$. 
We abbreviate the subschemes cut out by the multiplier ideal sheaves to the MIS.
Especially, for an almost plurisubharmonic function $\varphi$ we call the subscheme cut out by $\mathcal{I}(\gamma\varphi)$ the MIS of
exponent $\gamma$ (with respect to $\varphi$). 
The exponent of the MIS is closely related to the complex singularity exponent, which is
introduced by Demailly-Koll\'ar \cite{demailly-kollar01} and the definition of the complex singularity exponent will be explained in Section \ref{sec:MIS_1ps}.
Here let us remark that the complex singularity exponent is a local version of a holomorphic invariant which is called the $\alpha$-invariant defined by Tian \cite{tian87}. 
He proved that Fano manifolds admit K\"ahler-Einstein metrics when the $\alpha$-invariant is strictly greater than $n/(n+1)$ by using the continuity method.
On the other hand, the same result is observed in \cite{rubinstein0708} by using the K\"ahler-Ricci flow.
In fact, Theorem \ref{thm:PSS_Rubinstein} in \cite{rubinstein0708} is obtained by effectively proving that if $\alpha_G(X)>\frac{n}{n+1}$ then the K\"ahler-Ricci flow converges.
Remark that $\alpha_G(X)\ge 1$ if there is no multiplier ideal sheaf $\mathcal{I}(\psi)$ such that there is a positive constant $\varepsilon$ satisfying that $\mathcal{I}(\gamma\psi)$ is proper for $\gamma \in (1-\varepsilon, 1)$.
This result implies many examples of K\"ahler-Einstein manifolds and it has been studied well.
For example , see \cite{tian-yau87} for K\"ahler-Einstein Fano surfaces, \cite{batyrev-selivanova99}, \cite{song05} for toric Fano manifolds, \cite{donaldson0803}  for recent progress, \cite{heier0710}, \cite{chen-wang0809} for recent works related to the K\"ahler-Ricci flow, \cite{cheltsov-shramov-demailly0806} for the relation between the complex singularity exponent and the $\alpha$-invariant and the references therein.

The purpose of this paper is to calculate the support of the KRF-MIS on certain toric Fano manifolds with large symmetry.
Let us explain the class of toric Fano manifolds we shall consider.
Let $X$ be a toric Fano manifold with an effective action of $T_{\mathbb{C}}:=(\mathbb{C}^*)^n$, where $\dim_\mathbb{C}X=n$.
Let $T_\mathbb{R}:=(S^1)^n$ be the real torus of $T_{\mathbb{C}}$ and $\mathfrak{t}_{\mathbb{R}}$ be the associated Lie algebra.
Let $N_\mathbb{R}:=J\mathfrak{t}_\mathbb{R}\simeq \mathbb{R}^n$ where $J$ is the complex structure of $T_{\mathbb{C}}$. 
Let $M_\mathbb{R}$ be the dual space $Hom(N_\mathbb{R}, \mathbb{R})\simeq\mathbb{R}^n$ of $N_\mathbb{R}$.
For each toric manifold $X$, there is an associated convex polytope $P^*\subset M_\mathbb{R}$ which is the image of the moment map from $X$ to $M_\mathbb{R}$.
For $P^*$, we denote its dual polytope by $P\subset N_\mathbb{R}$, which is often called a Fano polytope.
The duality of $P$ and $P^*$ is defined by
\[
	P^*=\{y\in M_\mathbb{R} \mid \langle y, q^{(i)} \rangle \le 1
	\,\,\,\mbox{for all vertices } q^{(i)} \,\, \mbox{of } P\}.
\] 
Let $\mathcal{N}(T_{\mathbb{C}})$ be the normalizer of $T_{\mathbb{C}}$ in $\mbox{Aut}(X)$. Then the Weyl group $\mathcal{W}(X):=\mathcal{N}(T_{\mathbb{C}})/T_{\mathbb{C}}$ of $\mbox{Aut}(X)$ with respect to $T_{\mathbb{C}}$ equals to the finite subgroup of $\mbox{GL}(N, \mathbb{Z})$ consisting of all elements which preserve $P$ where $N\simeq \mathbb{Z}^n$ is the space of all lattice points in $N_\mathbb{R}$ (see Proposition 3.1 in \cite{batyrev-selivanova99}).
Let $N_{\mathbb{R}}^{\mathcal{W}(X)}:=\{x\in N_\mathbb{R} \mid x^g=x \,\,\,\mbox{for all } g\in \mathcal{W}(X)\}$.
Then, the class of toric Fano manifolds which we shall consider is
\[
	\mathcal{W}_1:=
	\{
		X:\,\, \mbox{toric Fano manifold with }
		\dim N_\mathbb{R}^{\mathcal{W}(X)}=1
	\}.
\]
The advantage to restrict the class of toric Fano manifolds to $\mathcal{W}_1$ is that it allows us to determine the holomorphic vector field of K\"ahler-Ricci solitons precisely only by the sign of its Futaki invariant and to calculate the KRF-MIS by using a K\"ahler-Ricci soliton,
although $\mathcal{W}_1$ might be quite limited.
Remark that Wang-Zhu \cite{wang-zhu04} proved that every toric Fano manifold has a K\"ahler-Ricci soliton.
We choose $G$ to be the maximal compact subgroup in $\mathcal{N}(T_{\mathbb{C}})$ generated by $T_{\mathbb{R}}$ and $\mathcal{W}(X)$ so that we have the short exact sequence
\[
	1\to T_\mathbb{R} \to G \to \mathcal{W}(X) \to 1.	
\]

Then our main result is as follows.
\begin{theorem}\label{thm:reduction_MIS}
Let $X$ be a toric Fano manifold in $\mathcal{W}_1$. 
Suppose that $X$ does not admit K\"ahler-Einstein metrics
and that the imaginary part $\mathfrak{Im}(v_{KRS})$ of $v_{KRS}$ generates a one-parameter subgroup of $G$. Let $\{\sigma_t:=\exp(tv_{KRS})\}$ and $\gamma \in (0,1)$. Then, the support of the KRF-MIS of
exponent $\gamma$ is equal to the support of the MIS of
exponent $\gamma$ derived from a sequence of K\"ahler potentials of $\{(\sigma_{t}^{-1})^*\omega\}$ for any $G$-invariant K\"ahler form $\omega$.
\end{theorem}
\begin{remark}
The author expects that the restriction to $\mathcal{W}_1$ would be just a technical assumption and it would be ruled out.
On the other hand, it is not known yet whether the restriction would imply that $X$ does not admit  K\"ahler-Einstein metrics.
This corresponds to a special case of the question in \cite{batyrev-selivanova99} which inquires whether all toric K\"ahler-Einstein Fano manifolds are symmetric or not.
Here recall that a toric Fano manifold $X$ is called symmetric if $\dim N_\mathbb{R}^{\mathcal{W}(X)}=0$.
\end{remark}
Theorem \ref{thm:reduction_MIS} says that the KRF-MIS on $X\in \mathcal{W}_1$ is reduced to the MIS derived from a one-parameter subgroup of the torus action.
In order to calculate the support of multiplier ideal subschemes on toric Fano manifolds, it is sufficient to calculate the complex singularity exponent with respect to the associated almost plurisubharmonic function for each face of the polytope $P^*\subset M_\mathbb{R}$.
Then we shall give a formula to calculate the complex singularity exponent of the MIS obtained from one-parameter subgroups of the torus action in Theorem \ref{thm:cse_face_formula_face}.
Combining Theorem \ref{thm:reduction_MIS} and Theorem {\ref{thm:cse_face_formula_face}}, we can calculate the support of the KRF-MIS concretely.
For example, we can prove
\begin{corollary}\label{cor:CP2_2pts}
Let $X$ be the blow up of $\mathbb{CP}^2$ at $p_1$ and $p_2$.
Let $E_1$ and $E_2$ be the exceptional divisors of the blow up, and $E_0$ be the proper transform of $\overline{p_1p_2}$ of the line passing through $p_1$ and $p_2$.
Then, the support of the KRF-MIS on $X$ of
exponent $\gamma$ is 
\[
	\left\{\begin{array}{cc}
		\cup_{i=0}^{2} E_i & \mbox{for } \gamma \in (\frac{1}{2}, 1),  \\
		E_0 & \mbox{for } \gamma \in (\frac{1}{3}, \frac{1}{2}).
	\end{array}\right.
\]
\end{corollary}

Finally, let us remark a relation with stability of manifolds.
As an analogue of slope stability of vector bundles, Ross-Thomas \cite{ross-thomas06} defined the slope for subschemes of a polarized manifold and proved that their slope stability is necessary for the existence of constant scalar curvature K\"ahler metrics.
From the viewpoint of Hitchin-Kobayashi correspondence, we expect that the KRF-MIS would destabilize a Fano manifold $(X, c_1(X))$ with anticanonical polarization.
Unfortunately, it is proved recently by Panov-Ross \cite{panov-ross0710} that the blow up of $\mathbb{CP}^2$ at two points is slope stable with respect to the anticanonical polarization, while it is not a K\"ahler-Einstein manifold.
On the other hand, by the formula (Corollary 5.3 in \cite{ross-thomas06}) to calculate the slope of smooth curves in a surface, we can see that $E_0$ in Corollary \ref{cor:CP2_2pts} has the worst slope.
This fact suits that $E_0$ has the worst complex singularity exponent in Corollary \ref{cor:CP2_2pts}.
In other words, our result suggests that the slope of subschemes would be related to the strength of singularity of the KRF-MIS.

The organization of this paper is as follows.
In Section \ref{sec:convergence} we shall reduce the KRF-MIS to a simpler one by following the proof of the convergence of the K\"ahler-Ricci flow by Tian-Zhu.
In Section \ref{sec:MIS_1ps} for each face of $P^*$ we shall give a formula to calculate the complex singularity exponent of the associated almost plurisubharmonic function derived from one-parameter subgroups of the torus action and complete the proof of the main theorem. Furthermore we shall give a way to determine the support of the KRF-MIS.
In Section \ref{sec:example} we shall calculate examples of toric Fano $n$-folds ($n=2,3$) contained in $\mathcal{W}_1$ by using our results.

\textit{Acknowledgement}: 
This work was supported by World Premier International Research Center Initiative (WPI Initiative), MEXT Japan while the author was a project researcher of the Institute for the Physics and Mathematics of the Universe (IPMU).
The author would like to thank Professor Akito Futaki for his valuable comments for the improvement of this paper.

\section{Convergence of the K\"ahler-Ricci flow to K\"ahler-Ricci solitons on toric Fano manifolds}
\label{sec:convergence}
Through this section and the next section, we shall prove the main theorem by using the results of Tian-Zhu \cite{tian-zhu06} and Zhu \cite{zhu0703} about the convergence of the K\"ahler-Ricci flow.
Firstly, let us recall toric Fano manifolds briefly. 
A toric variety $X$ is an algebraic variety with an effective action of $T_{\mathbb{C}}:=(\mathbb{C}^*)^n$, where $\dim_\mathbb{C}X=n$.
Let $T_\mathbb{R}:=(S^1)^n$ be the real torus in $T_{\mathbb{C}}$ and $\mathfrak{t}_{\mathbb{R}}$ be the associated Lie algebra.
Let $N_\mathbb{R}:=J\mathfrak{t}_\mathbb{R}\simeq \mathbb{R}^n$ where $J$ is the complex structure of $T_{\mathbb{C}}$. Let $M_\mathbb{R}$ be the dual space $Hom(N_\mathbb{R}, \mathbb{R})\simeq\mathbb{R}^n$ of $N_\mathbb{R}$. Denoting the group of algebraic characters of $T_{\mathbb{C}}$ by $M$, then $M_\mathbb{R}=M\otimes_\mathbb{Z} \mathbb{R}$.
It is well-known that for each smooth toric Fano manifold $X$ there is a fan $\Sigma_X$ such that
\begin{enumerate}
	\item[(a)]
		the polytope $P$ consisting of the set of the primitive elements of 
		all $1$-dimensional cones in $\Sigma_X$ is an $n$-dimensional convex polytope,
	\item[(b)]
		the origin of $N_\mathbb{R}$ is contained in the interior of $P$,
	\item[(c)]
		any face of $P$ is a simplex, and
	\item[(d)]
		the set of vertices of  any $(n-1)$-dimensional face of $P$ constitutes a basis 
		of $N\simeq\mathbb{Z}^n \subset N_\mathbb{R}$.
\end{enumerate}
The polytope $P$ is often called the Fano polytope of $X$.

Next let us recall the definition of K\"ahler-Ricci solitons. A pair $(v, \omega)$ of a holomorphic vector field and a K\"ahler form on a Fano manifold is called a K\"ahler-Ricci soliton if 
\[
	\mbox
	{Ric}(\omega)-\omega=\mathcal{L}_v\omega,
\]
where $\mathcal{L}_v$ is the Lie derivative along $v$. Obviously K\"ahler-Einstein metrics are K\"ahler-Ricci solitons with $v=0$. The existence of K\"ahler-Ricci solitons on toric Fano manifolds is proved by Wang-Zhu \cite{wang-zhu04}.
\begin{theorem}[Wang-Zhu, \cite{wang-zhu04}]\label{thm:wang-zhu}
There exists a K\"ahler-Ricci soliton, which is unique up to the identity component of the group of  holomorphic automorphisms, on a toric Fano manifolds.
\end{theorem}
In the recent progress of the study about the Ricci flow after Perelman's works, the convergence of the K\"ahler-Ricci flow on Fano manifolds with K\"ahler-Ricci solitons is proved by Tian-Zhu \cite{tian-zhu06}.
This is a generalization of the result announced by Perelman \cite{perelman} which says that if $X$ admits a K\"ahler-Einstein metric then the K\"ahler-Ricci flow will converge to a K\"ahler-Einstein metric in the sense of Cheeger-Gromov.
Let $\mbox{Aut}_r(X)$ be the reductive part of $\mbox{Aut}(X)$ and $K$ be a maximal compact subgroup of $\mbox{Aut}_r(X)$. Note that $\mbox{Aut}_r(X)$ is the complexification of $K$. From the uniqueness of K\"ahler-Ricci solitons proved by Tian-Zhu in \cite{tian-zhu02}, we may assume that a K\"ahler-Ricci soliton $(v_{KRS},\omega_{KRS})$ is $K$-invariant and the imaginary part of $v_{KRS}$ generates a one-parameter subgroup $K_{v_{KRS}}$ of $K$.
For a holomorphic vector field $v$, let $F_{v}$ be the holomorphic invariant defined by Tian-Zhu \cite{tian-zhu02} , which is a generalization of Futaki invariant. The definition of Futaki invariant will be explained in Section \ref{sec:example}.
Then the holomorphic vector field $v_{KRS}$ satisfies that  $F_{v_{KRS}}$ vanishes on $\mbox{Aut}_r(X)$.
\begin{theorem}[Tian-Zhu, \cite{tian-zhu06}]\label{thm:tian-zhu}
Let $X$ be a Fano manifold which admits a K\"ahler-Ricci soliton $(v_{KRS}, \omega_{KRS})$ as above.
Then, any solution $\omega_t$ of the normalized K\"ahler-Ricci flow (\ref{eq:ricci_flow}) will converge to $\omega_{KRS}$ in the sense of Cheeger-Gromov if the initial K\"ahler metric is $K_{v_{KRS}}$-invariant.
\end{theorem}
Combining Theorem \ref{thm:wang-zhu} and Theorem \ref{thm:tian-zhu}, we find that the normalized K\"ahler-Ricci flow (\ref{eq:ricci_flow}) will converge in the sense of Cheeger-Gromov on toric Fano manifolds.
The same result is proved by Zhu \cite{zhu0703} without the assumption of the existence of K\"ahler-Ricci solitons on toric Fano manifolds.
These results suggest us that the KRF-MIS would be calculated by using K\"ahler-Ricci solitons on toric Fano manifolds which do not admit K\"ahler-Einstein metrics. 
In fact, we shall see that this attempt works well on toric Fano manifolds with certain symmetry.
For this purpose, let us explain about symmetry of toric Fano manifolds (cf. \cite{batyrev-selivanova99}, \cite{song05}).
Let $\mathcal{N}(T_{\mathbb{C}})$ be the normalizer of $T_{\mathbb{C}}$ in $\mbox{Aut}(X)$. Then the Weyl group $\mathcal{W}(X):=\mathcal{N}(T_{\mathbb{C}})/T_{\mathbb{C}}$ of $\mbox{Aut}(X)$ with respect to $T_{\mathbb{C}}$ equals to the finite subgroup of $\mbox{GL}(N, \mathbb{Z})$ consisting of all elements which preserve $P$ where $N\simeq \mathbb{Z}^n$ is the dual of $M$ (Proposition 3.1 in \cite{batyrev-selivanova99}).
Let $N_{\mathbb{R}}^{\mathcal{W}(X)}:=\{x\in N_\mathbb{R} \mid x^g=x \,\,\,\mbox{for all } g\in \mathcal{W}(X)\}$.
Then, the class of toric Fano manifolds which we shall consider is
\[
	\mathcal{W}_1:=
	\{
		X:\,\, \mbox{toric Fano manifold with }
		\dim N_\mathbb{R}^{\mathcal{W}(X)}=1
	\}.
\]

Then we shall prove Theorem \ref{thm:reduction_MIS} in the rest of this section and the next section.
Theorem \ref{thm:reduction_MIS} follows essentially from the argument of Zhu in \cite{zhu0703} (also of Tian-Zhu in \cite{tian-zhu06}). 
To be comprehensive as possible as we can, we shall recall the outline of the proof of \cite{zhu0703}. 
The key point in their proof of \cite{tian-zhu06} and \cite{zhu0703} for us is how to modify the K\"ahler-Ricci flow to converge to a K\"ahler-Ricci soliton.
Let $\omega_0$ be an initial K\"ahler form which is $G$-invariant.
Let us consider the equation of (\ref{eq:ricci_flow}) whose initial condition $c_0=0$, i.e.,
\begin{eqnarray}
	\label{eq:ricci_flow_potential2}
		\left\{\begin{array}{l}
			\frac{\partial\phi_t}{\partial t}
			=
			\log \frac{\det (g_{i\bar{j}}+\phi_{i\bar{j}})}{\det (g_{i\bar{j}})}+\phi_t -h_0,
		\\
			\phi_0 \equiv 0.			
		\end{array}\right.
\end{eqnarray}
Remark that $c_0$ in (\ref{eq:ricci_flow_potential2}) is different from the initial constant in \cite{pss0611} and \cite{rubinstein0708}, but we shall see in the proof of Lemma \ref{lem:step1} that this difference does not affect  the KRF-MIS.  
As an initial K\"ahler form $\omega_0$ on $X$, we take a standard metric determined by the moment polytope $P^*$ as follows.
Let $(\frac{1}{2}x_1+\sqrt{-1}\theta_1, \ldots, \frac{1}{2}x_n+\sqrt{-1}\theta_n)$ be an affine logarithm coordinates on $T_{\mathbb{C}}=T_\mathbb{R} \times N_\mathbb{R}$, i.e., $t_i=\exp(\frac{1}{2}x_i+\sqrt{-1}\theta_i)$ where $t=(t_1,\cdots, t_n)\in T_{\mathbb{C}}$.
Let $\{p^{(i)}\}_{i=1,\ldots, m}$ be the set of all lattice points contained in $P^{*}\subset M_\mathbb{R}$, and $\langle \cdot, \cdot \rangle$ is the natural inner product on $M_\mathbb{R}\times N_\mathbb{R}$.
Then we let $\omega_0:=\frac{\sqrt{-1}}{2\pi}\partial\bar{\partial}u_0$ on a dense orbit of the action of $T_{\mathbb{C}}$ where the quotient of $u_0$ is a convex function on $N_\mathbb{R}$ defined by
\begin{equation}\label{eq:canonical_metric}
	u_0(x):=\log \biggl(
	\sum_{i=1}^{m} e^{\langle p^{(i)}, x \rangle}
	\biggr)
\end{equation}
and $x=(x_1, \ldots, x_n)\in N_\mathbb{R}$.
It is known that $\omega_0$ can be extended to a well-defined K\"ahler form on $X$.
In fact $\omega_0$ is the pull-back of the Fubini-Study form on $\mathbb{CP}^{m-1}$ with respect to the anticanonical embedding $X\hookrightarrow \mathbb{P}(H^0(X, K^{-1}_X)^*)$.
Remark that the image of the moment map $\mu: X \to M_\mathbb{R}$ with respect to $\omega_0$ equals to $P^*$.
Obviously $\omega_0$ and $u_0$ are $\mathcal{W}(X)$-invariant.
By Lemma 4.3 in \cite{song05}, we find that there are positive constants $c$ and $C$ such that
\begin{equation}\label{eq:estimate_volume_form}
	c \le
	e^{u_0}\det \biggl(
	\frac{\partial^2 u_0}{\partial x_i \partial x_j}
	\biggr)
	\le 
	C.
\end{equation}
From (\ref{eq:ricci_discrepancy}) and (\ref{eq:estimate_volume_form}), we can assume 
\[
	\det ((u_0)_{ij}) =\exp (-u_0-h_0).
\]
Since $h_0$ and $\phi_t$ in (\ref{eq:ricci_flow_potential2}) are also $T_\mathbb{R}$-invariant, then we can reduce (\ref{eq:ricci_flow_potential2}) to a real Monge-Amp\`ere equation
\begin{eqnarray}
	\label{eq:real_MA}
		\left\{\begin{array}{l}
			\frac{\partial u}{\partial t}
			=
			\log \det (u_{ij}) +u,
		\\
			u(0,\cdot)
			=
			u_0,
		\end{array}\right.
\end{eqnarray}
where $u_t=u(t,\cdot)=u_0+\phi_t$ 
on $N_\mathbb{R}$. 
Here we denote the reduced potential functions of $\omega_t$ on $N_\mathbb{R}$, which is the quotient of $\phi_t$ to $N_\mathbb{R}$, also by the same $\phi_t$ to avoid the complicacy of symbols.
Note that the quotient of $\phi_t$ to $N_\mathbb{R}$ is normalized by requiring that the image of the gradient map of $u_t$ in $M_\mathbb{R}$ equals to $P^*$.
For each $t$ let $h_t$ and $c_t$ be the normalized Ricci discrepancy and a constant defined by
\[
		\int_X e^{h_t} \omega_t 
	=
		\int_X \omega_0^n,
	\,\,\,
		h_t
	=	
		-\frac{\partial \phi_t}{\partial t}+c_t,
\]
where $\omega_t$ is the solution of the K\"ahler-Ricci flow (\ref{eq:ricci_flow}).
As for $h_t$ above, we refer the following lemma which is proved by Perelman.
\begin{lemma}[Perelman, see also \cite{sesum-tian05}]\label{lem:perelman}
\[
	|h_t|\le A,
\]
where $A$ is independent of $t$.
\end{lemma}
For each solution $u_t$ of (\ref{eq:real_MA}), let $\bar{u}_t:=u_t-c_t$ and $m_t:=\inf_{x\in \mathbb{R}^n}\bar{u}_t(x)$.
Let $x_t$ be the minimal point of $\bar{u}_t$, $\bar{\bar{u}}_t:=\bar{u}_t(\cdot + x_t)-m_t$ and $\bar{\phi}_t$ be $\bar{\bar{u}}_t -u_0$.
The existence of $x_t$ for each $t$ is assured as follows. Since $\phi_t$ is the quotient of the function over $X$, it it bounded over $N_\mathbb{R}$. Then the existence of $x_t$ is equivalent to the existence of the minimal point of $u_0$, which is assured because $u_0$ is approximated by linear functions near the infinity in $N_\mathbb{R}$.
In fact, for any vector $x\in N_\mathbb{R}$, we have
\begin{eqnarray*}
	0<s\max_i\langle p^{(i)}, x \rangle
	\le
	u_0(sx)
	\le
	s\max_i\langle p^{(i)}, x \rangle + m
\end{eqnarray*}
for all $s\in \mathbb{R}_{\ge 0}$ where $m$ is the number of lattice points contained in $P^*$.
From Lemma \ref{lem:perelman} and the similar argument in \cite{wang-zhu04},
\begin{proposition}[Lemma 2.1 \cite{zhu0703}, Proposition 3.1 \cite{zhu0703}]
\[
	|m_t|\le C, \,\,\, \|\bar{\phi}_t\|_{C^0} \le C,
\]
where $C$ is independent of $t$.
\end{proposition}
To get higher order estimate, we shall modify $\bar{\phi}_t$.
\begin{lemma}[Lemma 4.6 \cite{chen-tian06}, Lemma 4.1 \cite{zhu0703}]\label{lem:path_modification}
Let $i$ be any nonnegative integer. Then the distance between $x_i$ and $x_{i+1}$ are uniformly bounded, i.e., $|x_i-x_{i+1}|<C$.
\end{lemma}
By replacing the original $x_t$ by a straight line segment $\overline{x_{i}x_{i+1}}$ for each unit interval $[i, i+1]$, Lemma \ref{lem:path_modification} allows us to modify the family of points $\{x_t\} \subset N_{\mathbb{R}}$ to a new family $\{x_t^{'} \}$ satisfying
\begin{equation}\label{eq:smoothing}
	|x_t-x_t^{'}| \le C,
	\,\,\,
	\biggl| \frac{dx_t^{'}}{dt} \biggr| \le C.
\end{equation}
Under our assumption that $X$ is contained in $\mathcal{W}_1$, we can choose a simple $\{x'_t\}$ as follows.
Let  $\beta_{KRS}$ be the vector in $N_\mathbb{R}$ which induces the holomorphic vector field $v_{KRS}$ of the K\"ahler-Ricci soliton. 
More precisely, if $v_{KRS}^{\sharp}$ is the real vector field induced by $\beta_{KRS}$ then $v_{KRS}=\frac{1}{2}(v_{KRS}^{\sharp}-\sqrt{-1}(Jv_{KRS}^{\sharp}))$.
Since $\beta_{KRS}$ is $\mathcal{W}(X)$-invariant and $X\in \mathcal{W}_1$, the line $\{s\beta_{KRS} \mid s\in \mathbb{R}\}$ equals to the fixed subspace of $N_\mathbb{R}$ under the action of $\mathcal{W}(X)$. 
Since $u_t$ is also $\mathcal{W}(X)$-invariant, $\{x_t\}_t$ is contained in the line $\{s\beta_{KRS} \mid s \in \mathbb{R}\}$, that is to say, for each $t$ there is a constant $s_t \in \mathbb{R}$ such that $x_t=s_t\beta_{KRS}$.
This fact and (\ref{eq:smoothing}) allow us to assume that $x_t^{'}=s_t\beta_{KRS}$ and $|ds_t/dt|$ is uniformly bounded for all $t$.
This assumption will simplify the calculation of the MIS  later when we prove Lemma \ref{lem:step2}.

Let $\rho_t$ be a holomorphic transformation, which induces
the shift transformation on $N_\mathbb{R}$ defined by $x \mapsto x +s_t\beta_{KRS}$ for each $t$.
Let $\tilde{\phi}_t$ be a K\"ahler potential defined by
\begin{equation}\label{eq:definition_tilde_phi}
	\rho_t^*\omega_{\phi_t}
	=\omega_0+\frac{\sqrt{-1}}{2\pi}\partial\bar{\partial}\tilde{\phi}_t,
\end{equation}
which is what we desire.
Remark that $\tilde{\phi}_t$ is equal to $u_0(\cdot+s_t\beta_{KRS}) -u_0(\cdot)$ up to constant.
The ambiguity of an additive constant in (\ref{eq:definition_tilde_phi}) is removed by requiring 
\begin{equation}\label{eq:modified_ricci_flow}
	\frac{\partial \tilde{\phi}_t}{\partial t}
	=
	\log \frac{\det (g_{i\bar{j}}+\tilde{\phi}_{i\bar{j}})}{\det(g_{i\bar{j}})}
	+\tilde{v}_t(\tilde{\phi}_t)+\tilde{\phi}_t-\tilde{h}_0+\theta_{\tilde{v}_t}
\end{equation}
on $X$, where 
\[
	\tilde{v}_t:=\frac{d x_t^{'}}{dt}=\beta_{KRS}\cdot\frac{d s_t}{dt},
\] 
$\theta_{\tilde{v}_t}:=\tilde{v}_t(u_0)$, and $\tilde{h}_0$ is the renormalized function of $h_0$ satisfying
\begin{eqnarray}
	\nonumber
		\frac{1}{V}\int_X(\tilde{h}_0-\theta_{v_{KRS}})\omega_0^n
	&=&
		-\frac{1}{V}\int_{0}^{\infty}\int_X
		\|\bar{\partial}\frac{\partial \phi'_t}{\partial t}\|^2
	     \exp(\theta_{v_{KRS}}+v_{KRS}(\phi'_t)-t)
	\\
	\label{eq:normalization_h0}
	&&
		\wedge(\sigma_t^*\omega_{\phi_t})^n\wedge dt
\end{eqnarray}
as Lemma 4.2 \cite{tian-zhu06}.
In (\ref{eq:normalization_h0}) $V$ denotes the volume of $X$ with respect to $\omega_0$, $\theta_{v_{KRS}}=v_{KRS}(u_0)$ and $\phi^{'}_t$ is the K\"ahler potential defined by
\[
	\sigma_t^* \omega_{\phi_t}=\omega_0+\frac{\sqrt{-1}}{2\pi}\partial\bar{\partial}
	\phi_t^{'}
\]
and
\[
	\frac{\partial \phi'_t}{\partial t}
	=
	\log\frac{\det (g_{i\bar{j}}+(\phi'_t)_{i\bar{j}})}{\det (g_{i\bar{j}})}
	+v_{KRS}(\phi'_t)+\phi'_t-h_0+\theta_{v_{KRS}}.
\]
Then, Tian-Zhu \cite{tian-zhu06} (also \cite{zhu0703}) proved 
\begin{proposition}[\cite{tian-zhu06}, \cite{zhu0703}]\label{prop:convergence}
The family $\{\omega_{\tilde{\phi}_t}\}_t$ converges to a  K\"ahler-Ricci soliton associated to $v_{KRS}$ and $\tilde{v}_t$ converges to $v_{KRS}$ as $t$ goes to the infinity.
\end{proposition}
Remark that $\frac{d s_t}{dt}\to1$ as $t\to \infty$, because $\tilde{v}_t$ converges to $v_{KRS}$.
Therefore we can conclude the following lemma. 
\begin{lemma}\label{lem:step1}
The KRF-MIS equals to the MIS coming from a family of K\"ahler potentials $\psi_t$ of $\{(\rho_t^{-1})^*\omega\}_t$ with respect to a fixed K\"ahler form $\omega$, which is normalized by $\sup \psi_t=0$, where $\omega$ is any $G$-invariant K\"ahler form.
\end{lemma}
\begin{proof}
Firstly we shall see that the difference of the choice of initial constant $c_0$ does not matter when we consider the KRF-MIS in the sense of \cite{rubinstein0708}.
In fact, the difference between (\ref{eq:ricci_flow_potential}) and (\ref{eq:ricci_flow_potential2}) induces that $\phi_t=\varphi_t-c_0e^t$ where the constant $c_0$ is the initial condition in Theorem \ref{thm:PSS_Rubinstein}.
However, since $\phi_t-\sup\phi_t$ equals to $\varphi_t-\sup\varphi_t$ for each $t$, the MIS coming from a family $\{\phi_t-\sup\phi_t\}$ coincides with the MIS obtained from a family $\{\varphi_t-\sup\varphi_t\}$, which is equal to the KRF-MIS in Theorem \ref{thm:PSS_Rubinstein}.
Take any $G$-invariant K\"ahler form $\omega$.
As seen in the above argument, we find that $\omega_{\phi_t}$ equals to $(\rho_t^{-1})^{*}\omega_{\tilde{\phi}_t}$ for each $t$.
Let $\psi_t^{'}\in C^\infty(X)$ be the discrepancy function defined by  $\omega_{\tilde{\phi}_t}-\omega=\frac{\sqrt{-1}}{2\pi}\partial\bar{\partial}\psi_t^{'}$ and $\sup \psi_t^{'}=0$.
Since $\omega_{\tilde{\phi}_t}$ converges in $C^\infty$-sense, $\|\psi_t^{'}\|_{C^0}$ is uniformly bounded.
Since
\begin{eqnarray*}
		(\rho_t^{-1})^{*}\omega
	&=&
		(\rho_t^{-1})^{*}\omega_{\tilde{\phi}_t}-\frac{\sqrt{-1}}{2\pi}\partial\bar{\partial}
		(\rho_t^{-1})^{*}\psi_t^{'}
	\\
	&=&
		\omega_{\phi_t}-\frac{\sqrt{-1}}{2\pi}\partial\bar{\partial}
		(\rho_t^{-1})^{*}\psi_t^{'},
\end{eqnarray*}
then
\[
	\psi_t
	=
	(\phi_t-(\rho_t^{-1})^{*}\psi_t^{'})-\sup (\phi_t-(\rho_t^{-1})^{*}\psi_t^{'}).
\]
Since $\|(\rho_t^{-1})^{*}\psi_t^{'}\|_{C^0}$ is also uniformly bounded, the MIS obtained from $\{\psi_t\}$ equals to the MIS obtained from $\{\phi_t-\sup\phi_t\}$. Hence, the proof is completed.
\end{proof}
In order to finish the proof of Theorem \ref{thm:reduction_MIS}, it is sufficient to show
\begin{lemma}\label{lem:step2}
Let $\gamma \in (0,1)$.
The support of the MIS obtained from $\{\psi_t\}$ of
exponent $\gamma$
equals to the support of the MIS obtained from the family of the normalized K\"ahler potentials of $\{(\sigma_t^{-1})^{*}\omega\}_t$ with respect to $\omega$, whose supremum equals to zero, of
exponent $\gamma$.
\end{lemma}
We shall prove the above lemma in the next section.

\section{Complex singularity exponents of multiplier ideal sheaves on toric Fano manifolds}\label{sec:MIS_1ps}
In this section, we shall give a formula to calculate the complex singularity exponent of the limit of $\psi_t$ in Lemma \ref{lem:step2} with respect to each face of the polytope $P^*\subset M_\mathbb{R}$.
Then, we shall give a proof to Lemma \ref{lem:step2} and complete the proof of Theorem \ref{thm:reduction_MIS}.
Furthermore, we shall give a way to determine the support of the KRF-MIS on $X$ which does not admit K\"ahler-Einstein metrics and is contained in $\mathcal{W}_1$.

Firstly, let us recall the complex singularity exponent of plurisubharmonic functions, which is introduced by Demailly and Koll\'ar  in \cite{demailly-kollar01} to describe the singularity of plurisubharmonic functions numerically.
Remark that the definition explained below can be applied to almost plurisubharmonic functions in our case, because an almost plurisubharmonic function we shall consider is written locally as a sum of a plurisubharmonic function and a smooth function which is a potential of a fixed reference K\"ahler form.
Let $X$ be a complex manifold and $\varphi$ be a plurisubharmonic function on $X$. Let $K\subset X$ be a compact subset of $X$. The complex singularity exponent $c_K(\varphi)$ of $\varphi$ on $K$ is defined by
\[
	c_K(\varphi):=\sup
	\{
		c \ge 0;
		\exp(-c\varphi) 
		\,\,
		\mbox{is $L^1$ on a neighborhood of }
		K
	\}.
\]
If $\varphi \equiv -\infty$ near some connected component of $K$, we define $c_K(\varphi):=0$.
The complex singular exponent $c_K(\varphi)$ depends only on the behavior of $\varphi$ near its $-\infty$ poles.
From its definition,  $c_{\{p\}}(\varphi)$ is strictly less than some positive constant $\gamma$ if and only if the local section $1_{U_p}$ of $\mathcal{O}_X(U_p)$ is not contained in $\Gamma(U_p, \mathcal{I}(\gamma \varphi))$ for any open neighborhood $U_p$ at $p$, i.e., $p$ is contained in the support of the subscheme cut out by $\mathcal{I}(\gamma\varphi)$.
That is to say, the support of the MIS of
exponent $\gamma$ with respect to $\varphi$ is equal to
\[
	\{p \in X \mid c_{\{p\}}(\varphi) <\gamma\}.
\]

From now on, let $X$ be a toric Fano manifold whose K\"ahler class equals to $c_1(X)$.
Let $P\subset N_\mathbb{R}$ be the Fano polytope of $X$ and $P^*$ be the dual polytope which is the image of the moment map.
More precisely, $P^*$ is defined by
\[
	P^*=\{y\in M_\mathbb{R} \mid \langle y, q^{(i)} \rangle \le 1
	\,\,\,\mbox{for all vertices } q^{(i)} \,\, \mbox{of } P\}.
\] 
Let $\rho_t$ be a holomorphic transformation corresponding to change from $\omega_t$ to $\omega_{\tilde{\phi}_t}$, i.e.,
$\rho_t$ induces the shift on $N_\mathbb{R}$ defined by $x \mapsto x +s_t\beta_{KRS}$ for each $t$ as in the previous section.
Let $\omega_0$ be the standard K\"ahler form defined by (\ref{eq:canonical_metric}).
Let $\{\psi_t\}_t$ be the sequence of K\"ahler potentials of $\{(\rho^{-1}_t)^{*}\omega_0\}_t$ satisfying
\[
	(\rho^{-1}_t)^{*}\omega_0=\omega_0 +\frac{\sqrt{-1}}{2\pi}\partial\bar{\partial}
	\psi_t,
	\,\,\,
	\sup \psi_t=0
\]
as in the previous section.
Let $\psi_\infty$ be the almost plurisubharmonic function which is  the limit of $\{\psi_t\}_t$ in $L^1$-topology.

For a point $y\in M_\mathbb{R}$ we denote the complex singularity exponent of $\psi_\infty$ on $\mu^{-1}(y)$ by $c_{\{y\}}(\psi_\infty)$ where $\mu:X\to M_\mathbb{R}$ is the moment map with respect to $\omega_0$.
This notation makes sense.
In fact, $\mu^{-1}(y)$ is contained in the support of the MIS of
exponent $\gamma$ with respect to $\psi_\infty$ if and only if $c_{\{y\}}(\psi_\infty) <\gamma$, because the MIS on a toric manifold is $T_\mathbb{R}$-invariant and
\begin{equation}\label{eq:Borel-Lebesgue}
	c_{\{y\}}(\psi_\infty)=\inf_{p\in \mu^{-1}(y)}c_{\{p\}}(\psi_\infty).
\end{equation}
As for (\ref{eq:Borel-Lebesgue}), it is easy to check as follows.
It is trivial that $c_{\{y\}}(\psi_\infty)\le\inf_{p\in \mu^{-1}(y)}c_{\{p\}}(\psi_\infty)$ from the definition.
For any $c < \inf_{p\in \mu^{-1}(y)}c_{\{p\}}(\psi_\infty)$, there is an open covering $\cup_{p\in \mu^{-1}(y)}U_p$ of $\mu^{-1}(y)$ such that $U_p$ is an open neighborhood at $p$ and $e^{-c\psi_\infty}$ is integrable over $U_p$.
Since $\mu^{-1}(y)$ is compact, we find that $e^{-c\psi_\infty}$ is integrable over $\cup_{p\in \mu^{-1}(y)}U_p$, i.e., $c\le c_{\{y\}}(\psi_\infty)$.
Hence (\ref{eq:Borel-Lebesgue}) is proved.
For each face $\delta^*$ of $P^*$, let us calculate $c_{\{y\}}(\psi_\infty)$ where $y$ is a point in the relative interior of $\delta^*$.
In order to do it, we shall choose a reference point in the interior of $\delta^*$ as follows.
Let $\delta^*$ be an $(n-l-1)$-dimensional face of $P^*$. 
Let $\tilde{\mu}$ be the $G$-equivariant moment map from $N_\mathbb{R}$ to $M_\mathbb{R}$ with respect to $\omega_0$ defined by
\[
	\tilde{\mu}(x):=
	\biggl(
		\frac{\partial u_0}{\partial x_1}(x), \ldots, \frac{\partial u_0}{\partial x_n}(x)
	\biggr),
\]
where $u_0$ is defined by (\ref{eq:canonical_metric}).
Remark that the image of $\tilde{\mu}$ equals to the interior of  $P^*$.
From the duality between $P$ and $P^*$, for $\delta^*$ there is a unique $l$-dimensional face $\delta$ of $P$.
From the definition of $P$, $\delta$ is a simplex.
Let $\{q^{(i)}\}_{i=1,\ldots, l+1}$ be the set of vertices of $\delta$.
For $a_i \in \mathbb{R}_{>0}$ satisfying $\sum_{i=1}^{l+1}a_i=1$, we put $x^{(a)}:=a_1q^{(1)}+\cdots +a_{l+1}q^{(l+1)}$.
Obviously $x^{(a)}$ is contained in the relative  interior of $\delta$. Then
\begin{eqnarray}
	\nonumber
		\frac{\partial u_0}{\partial x_j}(sx^{(a)})
	&=&
		\frac{\partial}{\partial x_j}
		\biggl|_{x=sx^{(a)}}
		\log \biggl(
		\sum_{i=1}^{m} e^{\langle p^{(i)}, x \rangle}
		\biggr)
	\\
	\nonumber
	&=&	
		\frac{1}
		{
			\sum_{i=1}^{m} e^{\langle p^{(i)}, sx^{(a)} \rangle}
		}
		 \biggl\{
			\sum_{i=1}^{m} 
			p^{(i)}_j
			e^{\langle p^{(i)}, sx^{(a)} \rangle}
		\biggr\}
	\\
	\nonumber
	&=&
		\frac{1}
		{
			\biggl(
			\sum_{i_\alpha \in A} e^{(l+1)s}
			\biggr)
			+
			o(e^{(l+1)s})
		}
		 \biggl\{
		 	e^{(l+1)s}
			\bigl(
				\sum_{i_\alpha\in A} 
				p^{(i_\alpha)}_j
			\bigr)
		+
		o(e^{(l+1)s})
		\biggr\}
	\\
	\label{eq:limit_inner}
	&\to&
		\frac{\sum_{i_\alpha\in A}p^{(i_\alpha)}_j }{\sharp A}	
\end{eqnarray}
as $s\to \infty$, where $A$ is a subset of $\{1,\ldots, m\}$ such that $i_\alpha \in A$ if and only if $p^{(i_\alpha)}$ is contained in $\cap_{i=1}^{l+1}H_i$, where
$H_i:=\{y\in M_\mathbb{R}\mid \langle y, q^{(i)} \rangle =1\}$.
In the above $f(s)\in o(e^{cs})$ means $\lim_{s\to \infty}f(s)e^{-cs}=0$ and $\sharp A$ denotes the number of integers in $A$. 
The equation (\ref{eq:limit_inner}) means that the point 
\begin{equation}\label{eq:define_point}
	p^{(\delta^*)}:=\lim_{s\to \infty}\tilde{\mu}(sx^{(a)})
\end{equation}
is independent of the choice of a vector $a$ and is contained in the relative interior of the face $\delta^*$. 
In fact, $\{p^{(i_\alpha)}\}_{i_\alpha\in A}$ is the set of all integral points on $\delta^*$ and $p^{(\delta^*)}$ is the average of them.
So, in order to determine whether $\mu^{-1}(\delta^*)$ is contained in the MIS of
exponent $\gamma$ or not, it is sufficient to determine whether $c_{\{p^{(\delta^*)}\}}(\psi_\infty)$ is strictly smaller than $\gamma$ or not.
In fact, the $T_{\mathbb{C}}$-invariance of the MIS implies that if $p^{(\delta^*)}$ is contained in the MIS then $\delta^*$ is also contained in it.

Next, we shall give a formula to calculate $c_{\{p^{(\delta^*)}\}}(\psi_\infty)$ for each face $\delta^*$ of $P^*$.
Let $\{p^{(j_k)} \mid 1 \le j_k \le m,\,\, j_{k} < j_{k+1} \}$ be the subset of all integral points of $P^*$ satisfying
\begin{equation}\label{eq:p_max}
	\langle 
	p^{(j_k)}, -\beta_{KRS}\rangle=\max_{i=1,\ldots, m}\langle p^{(i)}, 
	-\beta_{KRS}
	\rangle.
\end{equation}
In order to distinguish such $p^{(j_k)}$ from the other integral points of $P^*$, we denote it by $p^{\max(k)}$.
Let $u^{'}_0(t, x)$ be a convex function on $N_\mathbb{R}$ defined by
\[
	u^{'}_0(t,x):=
	\log\biggl(
		\sum_{i=1}^{m}e^{\langle p^{(i)}, x-s_t\beta_{KRS} \rangle}
	\biggr)
	-
	s_t \max_{i=1,\ldots, m}\langle p^{(i)}, -\beta_{KRS} \rangle.
\]
Then, $(\rho_t^{-1})^*\omega_0=\frac{\sqrt{-1}}{2\pi}\partial\bar{\partial}u^{'}_0(t,x)$.
We find
\begin{eqnarray}
	\nonumber
		u^{'}_0(t,x)-u_0(x)
	&=&
		\log 
		\biggl(
		\frac{\sum_{i=1}^m e^{\langle p^{(i)}, x\rangle +
		s_t(\langle p^{(i)}, -\beta_{KRS}\rangle 
		-\max_{j}\langle p^{(j)}, -\beta_{KRS} \rangle)}}
		{\sum_{i=1}^m e^{\langle p^{(i)}, x\rangle} }
		\biggr)
	\\
	\label{eq:potential_less_than_zero}
	&\le&
	0
\end{eqnarray}
for all $x \in N_{\mathbb{R}}$ and all $t\in \mathbb{R}_{\ge 0}$, and on the other hand we also find 
\begin{eqnarray}
	\nonumber
	&&
		u^{'}_0(t, -s\beta_{KRS})-u_0(-s\beta_{KRS})
	\\
	\nonumber
	&=&
		\log
		\biggl(
			\frac{\sum_{i=1}^{m}
				e^{\langle p^{(i)}, -s\beta_{KRS}-s_t\beta_{KRS}\rangle
			-s_t \max(\langle p^{(i)}, -\beta_{KRS} \rangle)}}
			{\sum_{i=1}^{m}e^{\langle p^{(i)}, -s\beta_{KRS}\rangle}}
		\biggr)
	\\
	\nonumber
	&\ge&
		\log
		\biggl(
			\frac{\sharp \{p^{\max(k)}\} 
				\cdot e^{s\max(\langle p^{(i)}, -\beta_{KRS} \rangle)}}
			{\sum_{i=1}^{m}e^{\langle p^{(i)}, -s\beta_{KRS}\rangle}}
		\biggr)
	\\
	\label{eq:sup=0_1}
	&\to&
		0
\end{eqnarray}
as $s\to \infty$.
From (\ref{eq:potential_less_than_zero}) and (\ref{eq:sup=0_1}), we find
\begin{equation}
	\label{eq:sup=0_2}
		\lim_{s\to \infty}(u^{'}_0(t, -s\beta_{KRS})-u_0(-s\beta_{KRS}))=0.
\end{equation}
From (\ref{eq:potential_less_than_zero}) and (\ref{eq:sup=0_2}) we find  $\sup_{x\in N_\mathbb{R}} (u^{'}_0(t,x)-u_0(x))=0$, that is to say, $\psi_t(x)=u^{'}_0(t,x)-u_0(x)$.
\begin{theorem}\label{thm:cse_face_formula_face}
Let $\delta^*$ be an $(n-l-1)$-dimensional face of $P^*$. 
Let $\delta$ is the associated $l$-dimensional face of $P$ with $\delta^*$.
Then, we have the following two possibilities;
\begin{enumerate}
	\item[(i)]
If for any $x\in\delta$ there is an integral point $p^{\max(k) }_x$ of $P^*$ defined by (\ref{eq:p_max}), which might depend on $x$, such that $\langle p^{\max(k)}_x, x \rangle \ge 0$, then $c_{\{p^{(\delta^*)}\}}(\psi_\infty) \ge 1$.
In particular $\mu^{-1}(\delta^*)$ is not contained in the support of the $G^{\mathbb{C}}$-invariant MIS obtained from $\{\psi_t\}_t$ of
exponent $\gamma$ for any $\gamma<1$.
	\item[(ii)]
Suppose that there is a point $x \in \delta$ such that 
\begin{equation}\label{eq:max_k0_q}
	\langle p^{\max(k)}, x \rangle <0 \,\,\, \mbox{for any } k.
\end{equation}
Let $p^{\max(k_0)}$ be a vertex of $P^*$ and $x^{(0)}$  be a point in $\delta$ such that
\begin{equation}\label{eq:p_q_max_0}
	\langle p^{\max(k_0)}, x^{(0)} \rangle
	=
	\min_{x}\max_{k} \langle p^{\max(k)}, x \rangle
\end{equation}
where $x$ runs over 
\[
	\{x\in \delta \mid x \,\, \mathrm{satisfies}\,\, (\ref{eq:max_k0_q})\}.
\]
Then, we have
\[
	c_{\{p^{(\delta^*)}\}}(\psi_\infty)
	=
	\frac{1}
	{1  -  \langle p^{\max(k_0)}, x^{(0)} \rangle}<1.
\]
In particular $\mu^{-1}(\delta^*)$ is contained in the support of the $G^{\mathbb{C}}$-invariant MIS obtained from $\{\psi_t\}_t$ of
exponent $\gamma$ for any $\gamma \in (c_{\{p^{(\delta^*)}\}}(\psi_\infty), 1)$.
\end{enumerate}
\end{theorem}
\begin{proof}
Firstly, we shall show the case (i).
For any $x\in \delta$, the assumption implies
\begin{eqnarray}
	\nonumber
		u^{'}_0(t, sx)
	&=&
		\log\biggl(
			\sum_{i=1}^m e^{\langle p^{(i)}, sx -s_t\beta_{KRS} \rangle}
		\biggr)
		-
		s_t \langle p^{\max(k)}_x, -\beta_{KRS} \rangle
	\\
	\label{eq:estimate_u'0-2}
	&\ge&
		s\langle p^{\max(k)}_x, x \rangle
		\ge 0
\end{eqnarray}
for all $s\ge 0$.
Let
\[
	\tilde{U}:=
	\{
		s_1x+s_2\eta \in N_\mathbb{R} \mid
		x \in \delta,
		\,\,
		\eta \in N_\mathbb{R}, \, |\eta|=1,
		\,\,
		s_i \in \mathbb{R}_{\ge 0}, \, |s_2|<1
	\}.
\]
Let $U_{p^{(\delta^*)}} \subset X$ be the interior of $\mu^{-1}(\overline{\tilde{\mu}(\tilde{U})})$, where $\overline{\tilde{\mu}(\tilde{U})}$ denotes the closure of $\tilde{\mu}(\tilde{U})$ and
 $\mu:X\to M_\mathbb{R}$ and $\tilde{\mu}:N_\mathbb{R} \to M_\mathbb{R}$ are the moment maps with respect to $\omega_0$.
Then, from (\ref{eq:limit_inner}), we find that $U_{p^{(\delta^*)}}$ is an open neighborhood around $\mu^{-1}(p^{(\delta^*)})$.
Then, for $c\ge 0$, (\ref{eq:estimate_u'0-2}) implies
\begin{eqnarray}
	\label{eq:estimate2-1}
		\int_{U_{p^{(\delta^*)}}} e^{-c\psi_t}\omega_0^n
	&\le&
		C\int_{\tilde{U}}
		e^{-c\psi_t-u_0}dx_1\cdots dx_n
	\\
	\label{eq:estimate2-3}
	&=&
		C\int_{\tilde{U}}
		e^{-cu^{'}_0(t,x)+(-1+c)u_0(x)}dx_1\cdots dx_n	
	\\
	\nonumber
	&\le&
		C\int_{\tilde{U}}
		e^{(-1+c)u_0(x)}dx_1\cdots dx_n
	\\
	\label{eq:estimate2-2}
	&\le&				
		C
		\biggl(\int_{s=0}^{\infty}
		e^{(-1+c)s}ds
		\biggr)^{l+1}.
\end{eqnarray}
In (\ref{eq:estimate2-1}), we use the inequality (\ref{eq:estimate_volume_form}).
From (\ref{eq:estimate2-2}), we find that $\int_{U_{p^{(\delta^*)}}} e^{-c\psi_t}\omega_0^n$ is bounded if $0 \le c<1$.
Hence we find that $c_{\{p^{(\delta^*)}\}}(\psi_\infty)\ge 1$.

Next we shall prove the case (ii).
Before proving it, remark that the existence of the points $p^{\max(k_0)}$ and $x^{(0)}$ in (\ref{eq:p_q_max_0}) is assured.
In fact a function $x \mapsto \max_k\langle p^{\max(k)}, x \rangle$ is continuous on a compact set $\{x\in \delta \mid \langle p^{\max(k)}, x \rangle \le 0 \,\,\, \mbox{for all } k\}$  and it is
\[
	\left\{\begin{array}{ccc}
	\mbox{equal to zero} & \mbox{if }& \langle p^{\max(k)}, x \rangle =0
	\,\,\,
	\mbox{for some } k
	\\
	\mbox{strictly less than zero} & \mbox{if }& \langle p^{\max(k)}, x \rangle <0
	\,\,\,
	\mbox{for all } k.
	\end{array}\right.
\]
These mean that the minimal point $x^{(0)}$ of the above function is contained in $\{x\in \delta \mid \langle p^{\max(k)}, x \rangle < 0 \,\,\, \mbox{for all } k\}$.
Let us begin to prove (ii).
The definition (\ref{eq:p_q_max_0}) implies that for all $x\in \delta$
\begin{equation}\label{eq:u'0_estimate}
		u^{'}_0(t, sx)
	\ge
		s\max_{k}\langle p^{\max(k)}, x \rangle
	\ge
		s\langle p^{\max(k_0)}, x^{(0)} \rangle.		
\end{equation}
Since for any $x\in \delta$ there is a vertex $p$ of $P^*$ such that $\langle x, p \rangle =1$, then we have
\begin{equation}\label{eq:u0_estimate}
	u_0(sx) \ge
	s.
\end{equation}
As (\ref{eq:estimate2-3}), for $0\le c<1$, (\ref{eq:u'0_estimate}) and (\ref{eq:u0_estimate}) imply
\begin{eqnarray}
	\nonumber
		\int_{U_{p^{(\delta^*)}}} e^{-c\psi_t}\omega_0^n
	&\le&
		C\int_{\tilde{U}}
		e^{-cu^{'}_0(t,x)+(-1+c)u_0(x)}dx_1\cdots dx_n	
	\\
	\label{eq:estimate2-4}
		&\le&
		C
		\biggl(
		\int_{s=0}^{\infty}
		e^{s\{
			-1+c(1-\langle p^{\max(k_0)}, x^{(0)} \rangle)
		\}}ds
		\biggr)^{l+1}.
\end{eqnarray}
From (\ref{eq:estimate2-4}) we find that 
\begin{equation}\label{eq:cse_bigger-2}
	c_{\{p^{(\delta^*)}\}}(\psi_\infty)
	\ge
	\frac{1}{1-\langle p^{\max(k_0)}, x^{(0)} \rangle}.
\end{equation}

Next we shall prove $c_{\{p^{(\delta^*)}\}}(\psi_\infty) \le\frac{1}{1-\langle p^{\max(k_0)}, x^{(0)} \rangle}$.
For each integral point $p^{(i)}$ of $P^*$, let
\begin{eqnarray}
	\nonumber
		A_i(s)
	&:=&
		\langle
			p^{(i)}, sx^{(0)} -s_t\beta_{KRS}
		\rangle
		-s_t\langle
			p^{\max(k_0)}, -\beta_{KRS}
		\rangle.
\end{eqnarray}
Then by (\ref{eq:p_max}), for all $i=1, \ldots, m$, we have
\begin{eqnarray}
	\nonumber
		A_i(s)
	&=&
		s(\langle
			p^{\max(k_0)}, x^{(0)}
		\rangle
		+\langle
			p^{(i)} - p^{\max(k_0)}, x^{(0)}
		\rangle)
	\\
	\label{eq:A_i(s)_1}
	&&
		-s_t(
			\langle p^{(i)}, \beta{_{KRS}}  \rangle
			-
			\langle p^{\max(k_0)}, \beta_{KRS}  \rangle	
		)
	\\
	\nonumber
	&\le&
		s(\langle
			p^{\max(k_0)}, x^{(0)}
		\rangle
		+\langle
			p^{(i)} - p^{\max(k_0)}, x^{(0)}
		\rangle).
\end{eqnarray}
If $\langle p^{(i)}-p^{\max(k_0)}, x^{(0)} \rangle \le 0$, then we have
\[
	A_i(s)\le s\langle p^{\max(k_0)}, x^{(0)} \rangle
	\,\,\,
	\mbox{for all }
	s\ge 0.
\]
If $\langle p^{(i)}-p^{\max(k_0)}, x^{(0)} \rangle > 0$, then (\ref{eq:p_q_max_0}) implies that $p^{(i)} \notin\{p^{\max(k)}\}_k$.
Otherwise it contradicts to that $\langle p^{\max(k_0)}, x^{(0)} \rangle$ is a maximum among $\{\langle p^{\max(k)}, x^{(0)}\rangle\}_k$.
This and (\ref{eq:p_max}) imply that $\langle p^{(i)}, \beta_{KRS} \rangle - \langle p^{\max(k_0)}, \beta_{KRS} \rangle$ is \textit{strictly} bigger than zero.
From (\ref{eq:A_i(s)_1}) we find that
\[
	A_i(s)\le s\langle p^{\max(k_0)}, x^{(0)} \rangle
	\,\,\,
	\mbox{for all }
	s \in [0, s_tT^{'}_{i}],
\]
where
\[
	T^{'}_{i}:=\frac{\langle (p^{(i)}-p^{\max(k_0)}), \beta_{KRS} \rangle}
	{\langle (p^{(i)}-p^{\max(k_0)}), x^{(0)} \rangle}.
\]
Let $T^{'}:=\min \{T^{'}_i \mid i=1,\ldots, m\} >0.$
This constant depends only on $\beta_{KRS}$ and independent of $s$ and $i$.
Hence , for all $i=1,\ldots, m$
\begin{equation}\label{eq:estimate_A_2-3}
	A_i(s)\le
	s\langle p^{(\max(k_0))}, x^{(0)} \rangle
	\,\,\,
	\mbox{for all }
	 s \in [0, s_tT^{'}].
\end{equation}
Let $\tilde{U}_{\varepsilon}:=\{x\in N_\mathbb{R} \mid |x-sx^{(0)}|<\varepsilon, \,\, s\ge \frac{1}{\varepsilon} \}$.
For any open neighborhood $U^{'}$ of $\mu^{-1}(p^{(\delta^*)})$, there is a sufficiently small constant $\varepsilon>0$ such that $\tilde{\mu}(\tilde{U}_\varepsilon) \subset \mu(U')$.
In fact, for the point $x^{(0)}$ in (\ref{eq:p_q_max_0}) we have
\begin{eqnarray}
	\nonumber
		\frac{\partial u_0}{\partial x_j}(sx^{(0)}+\eta)
	&=&
		\frac{\partial}{\partial x_j}
		\biggl|_{x=sx^{(0)}+\eta}
		\log \biggl(
		\sum_{i=1}^{m} e^{\langle p^{(i)}, x \rangle}
		\biggr)
	\\
	\nonumber
	&=&	
		\frac{1}
		{
			\sum_{i=1}^{m} e^{\langle p^{(i)}, sx^{(0)+\eta} \rangle}
		}
		 \biggl\{
			\sum_{i=1}^{m} 
			p^{(i)}_j
			e^{\langle p^{(i)}, sx^{(0)+\eta} \rangle}
		\biggr\}
	\\
	\nonumber
	&\to&
		\frac{
			\sum_{i_\alpha\in A} 
			e^{\langle p^{(i_\alpha)}, \eta \rangle}
			 p^{(i_\alpha)}_j
		}
		{
			\sum_{i_\alpha \in A}
				e^{\langle p^{(i_\alpha)}, \eta \rangle}
		}
\end{eqnarray}
as $s \to \infty$, where $A$ is the subset of $\{1,\ldots, m\}$ defined by (\ref{eq:limit_inner}).
Since $A$ is independent of $\eta$, there is a positive constant $C$ independent of $\varepsilon$ and $\eta$ such that
\begin{eqnarray}
	&&
	\nonumber
	|\lim_{s\to \infty}\tilde{\mu}(sx^{(0)}+\eta)-p^{(\delta^*)}|^2	\\
	\nonumber
	&=&
	\sum_{j=1}^{n}
	\biggl|
		\frac{
			\sum_{i_\alpha\in A} 
			e^{\langle p^{(i_\alpha)}, \eta \rangle}
			 p^{(i_\alpha)}_j
		}
		{
			\sum_{k_\alpha \in A}
				e^{\langle p^{(k_\alpha)}, \eta \rangle}
		}
		-
		\frac{\sum_{i_\alpha \in A} p^{(i_\alpha)}_j}{\sharp A}
	\biggr|^2
	\\
	\nonumber
	&\le&
	C
	\sum_{j=1}^n
	\biggl|
		\sum_{i_\alpha \in A}
		\biggl(
			\sum_{k_\alpha \in A}
			(e^{\langle p^{(i_\alpha)}, \eta  \rangle}-
			e^{\langle p^{(k_\alpha)}, \eta \rangle})
		\biggr)
		p^{(i_\alpha)}_j
	\biggr|^2
	\\
	\label{eq:asymptotic_discrepancy}
	&\le&
	C \varepsilon
\end{eqnarray}
for any sufficiently small $\varepsilon>0$ and any $\eta \in N_\mathbb{R}$ with  $|\eta| <\varepsilon$.
From (\ref{eq:asymptotic_discrepancy}), we find that there is a positive constant $C$ independent of $s$ and $\eta$ such that
\begin{equation}
\label{eq:potential_discrepancy_s}
	|\tilde{\mu}(sx^{(0)}+\eta)-\tilde{\mu}(sx^{(0)})|
	\le 
	C \varepsilon
\end{equation}
for all $s\in \mathbb{R}_{\ge 0}$ and any $\eta$ with $|\eta| < \varepsilon$.
This implies that $\tilde{\mu}(\tilde{U}_\varepsilon) \subset \mu(U')$ for any sufficiently small $\varepsilon$.
Remark that $\tilde{\mu}(\tilde{U}_{\varepsilon})$ is not necessarily a neighborhood of $p^{(\delta^*)}$.
(For instance, when $\delta^*$ is a $0$-dimensional face, $\tilde{\mu}(sx^{(0)}+\eta)$ goes to the point $p^{(\delta^*)}$ for any $\eta$, because $\sharp A=1$.)
There is a positive constant $C_\varepsilon$ depending only on $\varepsilon$ such that, for any $x\in \tilde{U}_{\varepsilon}$ with $|x-sx^{(0)}|<\varepsilon$,
\begin{equation}\label{eq:u'_bdd_by_logA-i}
	u^{'}_0(t,x)
	\le
	u^{'}_0(t,sx^{(0)}) +C_\varepsilon
	=
	\log\biggl(
	\sum_{i}^{m} \exp A_i(s)
	\biggr) +C_\varepsilon.
\end{equation}
On the other hand, 
\begin{equation}\label{eq:u_bdd_by_s_loglatticenumber}
	u_0(sx)\le s+\log m,
\end{equation}
where $x\in \delta$ and $m$ is the number of lattice points in $P^*$.
From (\ref{eq:estimate_A_2-3}), (\ref{eq:u'_bdd_by_logA-i}) and (\ref{eq:u_bdd_by_s_loglatticenumber}) we find that for $0 \le c < 1$ and a fixed sufficiently small $\varepsilon$, 
\begin{eqnarray}
	\nonumber
		\int_{U^{'}} e^{-c\psi_t} \omega_0^n
	&\ge&
		C\int_{\mu^{-1}(\tilde{\mu}(\tilde{U}_\varepsilon))} 
		e^{-c\psi_t} \omega_0^n
	\\
	\nonumber
	&\ge&
		C\int_{\tilde{U_\varepsilon}}
		e^{-cu^{'}_0(t, x)+(-1+c)u_0}dx_1\cdots  dx_n
	\\
	\nonumber
	&\ge&
		C\int_{\frac{1}{\varepsilon}}^{s_tT^{'}}
		e^{-c \max_i A_i(s)+(-1+c)s}ds
	\\
	\nonumber
	&\ge&
		C\int_{\frac{1}{\varepsilon}}^{s_tT^{'}}
		e^{-cs \langle p^{\max(k_0)}, x^{(0)} \rangle +(-1+c)s}ds
	\\
	\label{eq:estimate2-5}
	&=&
		C\int_{\frac{1}{\varepsilon}}^{s_tT^{'}}
		e^{s\{c(1- \langle p^{\max(k_0)}, x^{(0)} \rangle) -1\}}ds.
\end{eqnarray}
If $c \ge \frac{1}{1- \langle p^{\max(k_0)}, x^{(0)} \rangle}$, the RHS of (\ref{eq:estimate2-5}) goes to $+\infty$ as $t\to \infty$, because $s_t$ goes to $+\infty$.
The definition of the semi-continuity of the complex singularity exponent (\cite{demailly-kollar01}) implies that
\begin{equation}\label{eq:cse_smaller-2}
	c_{\{p^{(\delta^*)}\}}(\psi_\infty) \le
	\frac{1}{1- \langle p^{\max(k_0)}, x^{(0)} \rangle}.
\end{equation}
Hence we get the desired equation from (\ref{eq:cse_bigger-2}) and (\ref{eq:cse_smaller-2}).
The proof is completed.
\end{proof}
\begin{remark}
Theorem \ref{thm:cse_face_formula_face} is kind of local version of Song's formula  \cite{song05} of the $\alpha$-invariant on toric Fano manifolds.
\end{remark}
Then, Lemma \ref{lem:step2} is a corollary of Theorem \ref{thm:cse_face_formula_face}.
\begin{proof}[Proof of Lemma \ref{lem:step2}]
Theorem \ref{thm:cse_face_formula_face} still holds if we assume that $s_t\equiv t$. 
This means that the complex singularity exponent with respect to $\rho_t$ defined in the previous section equals to the one with respect to $\sigma_t$.
Therefore, the MIS obtained from $\{(\rho_t^{-1})^*\omega\}_t$ has the same support of  the MIS obtained from $\{(\sigma_t^{-1})^*\omega\}_t$ for any 
exponent $\gamma<1$.
\end{proof}
Therefore the proof of Theorem \ref{thm:reduction_MIS} is completed.

We shall conclude this section with another Corollary of Theorem \ref{thm:cse_face_formula_face}.
Let $\varepsilon>0$ be a sufficiently small constant.
Theorem \ref{thm:cse_face_formula_face} gives us a way to determine the support of the MIS of 
exponent $\gamma$ from any one-parameter subgroup of $\mbox{Aut}(X)$ for $\gamma \in (1-\varepsilon, 1)$ as follows. 
Here we do not need the assumption that $X$ is contained in $\mathcal{W}_1$ .
To describe the statement, let us introduce some terminologies.
Let ${\sigma_t}$ be a one-parameter subgroup of the holomorphic vector field $v_{\zeta}$ which is associated with a vector $\zeta \in N_\mathbb{R}$, i.e., if $\zeta^{\sharp}$ is the real vector field induced by $\zeta$ then $v_{\zeta}=\frac{1}{2}(\zeta^{\sharp}-\sqrt{-1}(J\zeta^{\sharp}))$ and $\sigma_t=\exp (tv_{\zeta})$.
Let us consider the MIS coming from $\{(\sigma_t^{-1})^*\omega_0\}_t$ as before.
Let $x(-\zeta)\in \partial P$ be a point which is the intersection between $\partial P$ and the half line $\{-s\zeta \in N_\mathbb{R} \mid s\ge0\}$. 
For distinct points $x^{(1)}$ and $x^{(2)}$ on $\partial P$, we define that $x^{(1)} \sim x^{(2)}$ if and only if $x^{(1)}$ and $x^{(2)}$ are contained in a common $(n-1)$-dimensional facet of $P$.
We define the star set of $x(-\zeta)$ by
\[
	st(x(-\zeta)):=
	\{
		x\in \partial P \mid
		x \sim x(-\zeta)
	\}.
\]
From the definition of the star set, $st(x(-\zeta))$ is a union of $(n-1)$-dimensional facets $\{\delta_k\}_{k=1,\ldots, k_\zeta}$ of $P$.
For each $\delta_k$, there corresponds to a hyperplane $\{x \in N_\mathbb{R} \mid H_k(x)=1\}$ in $N_\mathbb{R}$ which contains $\delta_k$.
Then, the star set $st(x(-\zeta))$ divides $N_\mathbb{R}$ into two.
This means that $N_\mathbb{R}$ is divided into $N_\mathbb{R}^{\le}:=\{x \mid H_k(x)\le 1 \,\,\, \mbox{for all } k\}$ and its complement.
Then, by translating $N_\mathbb{R}^{\le}$ along the line $\{-s\zeta \in N_\mathbb{R} \mid s \in \mathbb{R}\}$ so that the origin is contained in its boundary, we define a subspace in $N_\mathbb{R}$ by
\begin{eqnarray}
	\nonumber
		\widetilde{st(x(-\zeta))}
	&:=&
		\{
			x\in N_\mathbb{R}
			\mid
				H_k(x)\le 0 \,\,\, \mbox{for all }k
		\}.
\end{eqnarray}
\begin{corollary}\label{cor:star_cor}
Let $X$ be a toric Fano manifold.
Let ${\sigma_t}$ be a one-parameter subgroup of the holomorphic vector field $v_{\zeta}$ which is associated with a vector $\zeta \in N_\mathbb{R}$.
Suppose $\gamma \in (1-\varepsilon, 1)$ where $\varepsilon$ is a sufficiently small positive constant.
Let $\delta^*$ be an $(n-l-1)$-dimensional face of $P^*$ and $\delta$ be its associated $l$-dimensional face of $\delta$.
Then, $\delta^*$ is contained in the image of  the support of the MIS of
exponent $\gamma$ from $\{(\sigma_t^{-1})^* \omega_0\}_t$ under the moment map $\mu$ if and only if $\delta\cap int(\widetilde{st(x(-\zeta))}) \neq \emptyset$, where $int(\widetilde{st(x(-\zeta))})$ is the interior of $\widetilde{st(x(-\zeta))}$.
\end{corollary}
\begin{proof}
From the duality of $P$ and $P^*$, we find that for each $H_k$ there corresponds to $p^{\max(k)}$ defined as (\ref{eq:p_max}) and that $st(x(-\zeta))$ equals to
\[
	\{
		x\in N_\mathbb{R}
		\mid
		\langle
			p^{\max(k)}, x
		\rangle
		=1 \,\,\,
		\mbox{for all }k
	\}.
\]
Hence we find that $int(\widetilde{st(x(-\zeta))})$ equals to
\[
	\{
		x\in N_\mathbb{R}
		\mid
		\langle
			p^{\max(k)}, x
		\rangle
		<0 \,\,\,
		\mbox{for all }k
	\}.
\]
If $\delta\cap int(\widetilde{st(x(-\zeta))}) \neq \emptyset$, then Theorem \ref{thm:cse_face_formula_face} (ii) implies
$c_{(p^{\delta^*})}(\psi_\infty) < 1-\varepsilon(\delta) $ for some $\varepsilon(\delta)>0$ which might depend on $\delta$.
By taking a sufficiently small $\varepsilon$, we get that $c_{(p^{\delta^*})}(\psi_\infty) < 1-\varepsilon$ if $\delta\cap int(\widetilde{st(x(-\zeta))}) \neq \emptyset$, because the number of faces in $P$ is finite.
On the other hand, if $\delta\cap int(\widetilde{st(x(-\zeta))}) = \emptyset$, then Theorem \ref{thm:cse_face_formula_face} (i) implies
$c_{(p^{\delta^*})}(\psi_\infty) \ge 1 \ge 1-\varepsilon$.
This completes the proof.
\end{proof}

\section{Examples}\label{sec:example}
In this section, we shall calculate several examples which are contained in $\mathcal{W}_1$.
Let $v_{KRS}$ be the holomorphic vector field of K\"ahler-Ricci soliton, which is contained in the reductive part $\mathfrak{h}_r(X)$ of the Lie algebra $\mathfrak{h}(X)$ consisting of all holomorphic vector fields on $X$.
Since manifolds are contained in $\mathcal{W}_1$, we can determine the vector $\beta_{KRS}$ in $N_\mathbb{R}$ which induces $v_{KRS}$ by calculating the sign of its Futaki invariant.
Let us recall the definition of Futaki invariant (\cite{futaki83}).
Futaki introduced an integral invariant, which is a Lie character of $\mathfrak{h}(X)$, defined by
\[
	F(v):=\int_X vh_g \omega_g^n.
\]
He proved that $F$ is independent of the choice of $g$.
Let $\theta_{KRS} \in C^{\infty}(X)$ be a function defined by
\[
	\iota_{(v_{KRS})}\omega_{KRS}=
	\frac{\sqrt{-1}}{2\pi}\bar{\partial}
	\theta_{KRS},
	\,\,\,
	\int_X e^{\theta_{KRS}}\omega_{KRS}^n=\int_X \omega_{KRS}^n.
\]
Note that the existence and the uniqueness of $\theta_{v_{KRS}}$ are assured by the Hodge theory, because $\iota_{(v_{KRS})}\omega_{KRS}$ is a $\bar\partial$-closed $(0,1)$-form and there is no harmonic $1$-form on $X$ due to $c_1(X)>0$.
Since $(v_{KRS}, \omega_{KRS})$ is a K\"ahler-Ricci soliton, we find
\begin{eqnarray}
	\nonumber
		\frac{\sqrt{-1}}{2\pi}\partial\bar{\partial}h_{g_{KRS}}
	&=&
		\mbox{Ric}(\omega_{KRS})-\omega_{KRS}
	\\
	\nonumber
	&=&
		\mathcal{L}_{v_{KRS}}\omega_{KRS}=\frac{\sqrt{-1}}{2\pi}\partial\bar{\partial}\theta_{KRS}.
\end{eqnarray}
Remark that $\mathcal{L}_{v_{KRS}}\omega_{KRS}$ is a real $(1,1)$-form, because the imaginary part of $v_{KRS}$ is a Killing vector field.
So, we find that $h_{g_{KRS}}$ is equal to $\theta_{KRS}$ and that
\begin{equation}\label{eq:sign_KRS}
	F(v_{KRS})=\int_X v_{KRS}\theta_{KRS}\omega_{KRS}^n
	=\int_X |\bar{\partial}\theta_{KRS}|^2\omega_{KRS}^n >0.
\end{equation}
Then, in order to determine $\beta_{KRS}$ under the assumption that $X$ is contained in $\mathcal{W}_1$, it is sufficient to calculate the sign of Futaki invariant of the holomorphic vector field coming from a vector in $N_\mathbb{R}$, which is invariant under $\mathcal{W}(X)$.
To calculate the sign of Futaki invariant of holomorphic vector fields in the center of $\mathfrak{h}_r(X)$, we shall use the following result;
\begin{theorem}[Mabuchi, \cite{mabuchi87}]\label{thm:mabuchi}
Let $\mathbb{F}:=(F(t_1\frac{\partial}{\partial t_1}), \ldots, F(t_n\frac{\partial}{\partial t_n}))\in \mathbb{R}^n$.
Remark that $t_i\frac{\partial}{\partial t_i}$ is a $T_\mathbb{C}$-invariant holomorphic vector field on $T_\mathbb{C}$, which can be extended on $X$.
 Let $b(P^*) \in M_\mathbb{R}$ be the barycenter of $P^*$, i.e.,
\[
	\frac{1}{\int_{P^*}dy}(\int_{P^*}y_1 dy, \ldots, \int_{P^*} y_ndy),
\]
where $dy=dy_1\wedge \cdots \wedge dy_n$.
Then $\mathbb{F}$ equals to $-b(P^*)$.
\end{theorem}
The minus sign of $b(P^*)$ above comes from that  our choice of affine logarithmic coordinates has the opposite sign to the one in \cite{mabuchi87}.
Combining (\ref{eq:sign_KRS}) and Theorem \ref{thm:mabuchi} we find
\begin{equation}\label{eq:criterion_KRS}
	\langle b(P^*), \beta_{KRS} \rangle <0.
\end{equation}

\subsection{Toric Fano $2$-folds}
There are five types of toric Fano $2$-folds;
$\mathbb{CP}^2$, $\mathbb{CP}^1 \times \mathbb{CP}^1$ and the blow up of $\mathbb{CP}^2$ at $k$ points, where $k=1,2,3$.
K\"ahler-Einstein manifolds among them are $\mathbb{CP}^2$, $\mathbb{CP}^1 \times \mathbb{CP}^1$
and the blow up of $\mathbb{CP}^2$ at $3$ points.
Meanwhile the blow up of $\mathbb{CP}^2$ at $k$ points ($k=1,2$) does not admit K\"ahler-Einstein metrics and it is contained in $\mathcal{W}_1$. So we can apply our results to them.
Firstly let us consider the blow up of $\mathbb{CP}^2$ at one point.
\begin{example}\label{ex:blw_CP2_1pt}
The support of the KRF-MIS on $\mathbb{CP}^2 \sharp \overline{\mathbb{CP}^2}$ of exponent $\gamma$ is the exceptional divisor for all $\gamma \in (\frac{1}{2}, 1)$.
\end{example}
\begin{proof}
The polytope in $N_\mathbb{R}$ whose vertices are
\[
	(-1, -1), (-1, 0), (0, -1), (1,1),
\]
corresponds to the Fano polytope $P$ of $\mathbb{CP}^2 \sharp \overline{\mathbb{CP}^2}$.
Then, $N^{\mathcal{W}(X)}_\mathbb{R}$ is the one-dimensional subspace of $N_\mathbb{R}$ generated by a vector $(-1,-1)$.
From the symmetry of $P$, we find that $\beta_{KRS}$ is proportional to $(-1,-1)$.
Since the vertices of the polytope $P^*$ are
\[
	(-1, 0), (0,-1), (2, -1), (-1, 2),
\]
it is easy to see that $\langle b(P^*), (1,1) \rangle >0$. 
Then, (\ref{eq:criterion_KRS}) implies that $\beta_{KRS}=\beta(-1,-1)$ where $\beta >0$.
Also we find that
\[
	\widetilde{st(x(-\beta_{KRS})})=
	\{x=(x_1, x_2) \mid x_1-2x_2 \ge 0, 2x_1-x_2 \le 0  \}.
\]
The vertex of $P$ contained in $int(\widetilde{st(x(-\beta_{KRS})}))$ is $(-1, -1)$ which represents the exceptional divisor.
Then, Corollary \ref{cor:star_cor} implies that the support of the KRF-MIS of exponent $\gamma$ is the exceptional divisor where $\gamma$ is strictly smaller than $1$ and sufficiently close to $1$.
The subset $\{p^{\max(k)}\}$ of vertices of $P^*$ is 
\[
	\{(2,-1), (-1, 2)\}.
\]
For the facet $\delta^*$ of $P^*$ associated with the vertex $(-1, -1)$ of $P$, 
\[
	\langle p^{\max(k_0)}, x^{(0)} \rangle
	=\langle (2,-1), (-1, -1) \rangle
	=\langle (-1, 2), (-1, -1) \rangle
	=-1.
\]
Hence, 
\[
	c_{\{p^{(\delta^{*})}\}}(\psi_\infty)=c_{\{(-\frac{1}{2}, -\frac{1}{2})\}}(\psi_\infty)
	=\frac{1}{2}.
\]
Therefore the proof is completed.
\end{proof}
Next let us consider the blow up of $\mathbb{CP}^2$ at $p_1$ and $p_2$.
Let $E_1$ and $E_2$ be the exceptional divisors of the blow up.
In $X$, there is another $(-1)$-curve denoted by $E_0$, which intersects with $E_1$ and $E_2$
Remark that $E_0$ is the proper transform of $\overline{p_1p_2}$ of the line passing through $p_1$ and $p_2$.
Then,
\begin{example}\label{ex:blw_CP2_2pts}
The support of the KRF-MIS on $\mathbb{CP}^2 \sharp 2\overline{\mathbb{CP}^2}$ of exponent $\gamma$ is 
\[
	\left\{\begin{array}{cc}
		\cup_{i=0}^{2} E_i & \mbox{for } \gamma \in (\frac{1}{2}, 1),  \\
		E_0 & \mbox{for } \gamma \in (\frac{1}{3}, \frac{1}{2}).
	\end{array}\right.
\]
\end{example}
\begin{proof}
The polytope in $N_\mathbb{R}$ whose vertices are
\[
	(-1, 0) , (0, -1), (1,0), (1,1), (0,1),
\]
corresponds to the Fano polytope $P$ of $\mathbb{CP}^2 \sharp 2\overline{\mathbb{CP}^2}$.
Then, $N^{\mathcal{W}(X)}_\mathbb{R}$ is the one-dimensional subspace of $N_\mathbb{R}$ generated by a vector $(1,1)$.
From the symmetry of $P$, we find that $\beta_{KRS}$ is proportional to $(1,1)$.
Since the vertices of the polytope $P^*$ are
\[
	(-1,-1), (-1, 1), (0,1), (1, 0), (1, -1),
\]
we find that $\langle b(P^*), (1,1) \rangle <0$. 
Then, (\ref{eq:criterion_KRS}) implies that $\beta_{KRS}=\beta(1,1)$ where $\beta >0$.
Also we find that
\[
	\widetilde{st(x(-\beta_{KRS})})=
	\{x=(x_1, x_2) \mid x_1+x_2  \ge 0  \}.
\]
The vertices of $P$ contained in $int(\widetilde{st(x(-\beta_{KRS})}))$ are $(1, 0), (0,1)$, which represent the exceptional divisors $E_1$ and $E_2$,  and $(1,1)$ which represents the proper transform $E_0$.
Then, Corollary \ref{cor:star_cor} implies that the support of the KRF-MIS of exponent $\gamma$ is the sum of $E_0$, $E_1$ and $E_2$  where $\gamma$ is strictly smaller than $1$ and sufficiently close to $1$.
The subset $\{p^{\max(k)}\}$ of vertices of $P^*$ is 
\[
	\{(-1,-1)\}.
\]
For the facets $\eta_i^*, (i=1,2)$ of $P^*$ associated with the vertices $(1, 0)$ and $(0,1)$ of $P$  respectively, 
\[
	\langle p^{\max(k_0)}, x^{(0)} \rangle
	=\langle (-1,-1), (1, 0) \rangle
	=\langle (-1, -1), (0, 1) \rangle
	=-1.
\]
Hence, 
\[
	c_{\{p^{(\eta_1^{*})}\}}(\psi_\infty)=c_{\{(1, -\frac{1}{2})\}}(\psi_\infty)
	=\frac{1}{2}.
\]
Also $c_{\{p^{(\eta_2^{*})}\}}(\psi_\infty)=\frac{1}{2}$.
For the facet $\delta^*$ associated with the vertex $(1,1)$ of $P$,
\[
	\langle p^{\max(k_0)}, x^{(0)} \rangle
	=\langle (-1,-1), (1, 1) \rangle
	=-2.
\]
Hence, 
\[
	c_{\{p^{(\delta^{*})}\}}(\psi_\infty)=c_{\{(\frac{1}{2}, \frac{1}{2})\}}(\psi_\infty)
	=\frac{1}{3}.
\]
Therefore the proof is completed.
\end{proof}

\subsection{Toric Fano $3$-folds}
Toric Fano $3$-folds are classified completely (Remark 2.5.10 in \cite{batyrev99}).
According to the classification, there are eighteen types of toric Fano $3$-folds.
Five of them are K\"ahler-Einstein manifolds, and
eight of them are contained in $\mathcal{W}_1$ and do not admit K\"ahler-Einstein metrics. 
(As for the classification of K\"ahler-Einstein toric $3$-folds, see \cite{mabuchi87}.)
\begin{example}\label{ex:B1}
	Let $\mathcal{B}_1$ be $\mathbb{P}_{\mathbb{CP}^2}(\mathcal{O}\oplus \mathcal{O}(2))$.
	The support of the KRF-MIS on $\mathcal{B}_1$
	of exponent $\gamma$ is
	$S_\infty$
	for $\gamma \in (\frac{1}{2}, 1)$.
	Here $S_\infty$ is the divisor defined by a section 
	$(0, \sigma)$ of
	$\mathcal{O}\oplus \mathcal{O}(2)$ over $\mathbb{CP}^2$.
	More precisely, $S_\infty$ is the closure of
	\[
		\{[0; \sigma(p)] \in \mathcal{B}_1 \mid \sigma(p)\neq 0\}.
	\]
	Remark that it is not an exceptional divisor.
\end{example}
\begin{proof}
The vertices of the Fano polytope of $\mathcal{B}_1$
is
\[
	({}^tq^{(1)}, {}^tq^{(2)}, {}^tq^{(3)}, {}^tq^{(4)}, {}^tq^{(5)})
	=
	\left(\begin{array}{ccccc}
		1 & -1 & 1 & 1 &  0 \\
		0 &  0 & 1 & 0 & -1 \\
		0 &  0 & 0 & 1 & -1
	\end{array}\right),
\]
where $^tq^{(i)}$ denotes the transposition of the vector $q^{(i)}$.
This toric Fano manifold has a symmetry which permutes $\{q^{(3)}, q^{(4)}, q^{(5)}\}$, then it is contained in $\mathcal{W}_1$ and $N^{\mathcal{W}(\mathcal{B}_1)}_\mathbb{R}$ is generated by a vector $(1, 0, 0)$.
The vertices of the polytope $P^*$ is
\[
	\left(\begin{array}{cccccc}
		 1 & 1 &  1 & -1 & -1 & -1\\
		 0 & -1&  0 &  2 & -3 &  2\\
		 0 & 0 & -1 &  2 &  2 & -3
	\end{array}\right).
\]
From (\ref{eq:criterion_KRS}), we find that $\beta_{KRS}=\beta(1,0,0)$, where $\beta>0$.
The vertex of $P$ contained in $int(\widetilde{st(x(-\beta_{KRS})}))$ is
$\{q^{(1)}\}$, 
which represents $S_\infty$.
Then, Corollary \ref{cor:star_cor} implies that the support of the KRF-MIS of 
exponent $\gamma$ is 
$S_\infty$
where $\gamma$ is strictly smaller than $1$ and sufficiently close to $1$.
Its complex singularity exponent is
$\frac{1}{2}$.
Therefore the proof is completed.
\end{proof}
\begin{example}\label{ex:B2}
	Let $\mathcal{B}_2$ be
	 $\mathbb{P}_{\mathbb{CP}^2}(\mathcal{O}\oplus \mathcal{O}(1))$, 
	 which is the blow up of $\mathbb{CP}^3$ at one point.
	The support of the KRF-MIS on $\mathcal{B}_2$ of 
	exponent $\gamma$ is
	the divisor $S_\infty$ defined by a section $(0,\sigma)$ of 
	$\mathcal{O}\oplus \mathcal{O}(1)$ over $\mathbb{CP}^2$
	for $\gamma \in (\frac{1}{2}, 1)$.
	In this case $S_\infty$ is the exceptional divisor.
\end{example}
\begin{proof}
The vertices of the Fano polytope of $\mathcal{B}_2$
is
\[
	({}^tq^{(1)}, {}^tq^{(2)}, {}^tq^{(3)}, {}^tq^{(4)}, {}^tq^{(5)})
	=
	\left(\begin{array}{ccccc}
		1 & -1 & 0 & 0 &  1 \\
		0 &  0 & 1 & 0 & -1 \\
		0 &  0 & 0 & 1 & -1
	\end{array}\right).
\]
This toric Fano manifold has a symmetry which permutes $\{q^{(3)}, q^{(4)}, q^{(5)}\}$, then it is contained in $\mathcal{W}_1$ and $N^{\mathcal{W}(\mathcal{B}_2)}_\mathbb{R}$ is generated by a vector $(1, 0, 0)$.
The vertices of the polytope $P^*$ is
\[
	\left(\begin{array}{cccccc}
		 1 &  1 &  1 & -1 & -1 & -1\\
		 1 &  1& -1 &  1 & -3 &  1\\
		 1 & -1&  1 &  1 &  1 & -3
	\end{array}\right).
\]
From (\ref{eq:criterion_KRS}), we find that $\beta_{KRS}=\beta(1,0,0)$, where $\beta>0$.
The vertex of $P$ contained in $int(\widetilde{st(x(-\beta_{KRS})}))$ is
$\{q^{(1)}\}$, 
which represents the exceptional divisor.
Then, Corollary \ref{cor:star_cor} implies that the support of the KRF-MIS of
exponent $\gamma$ is 
the exceptional divisor
where $\gamma$ is strictly smaller than $1$ and sufficiently close to $1$.
Its complex singularity exponent is
$\frac{1}{2}$.
Therefore the proof is completed.
\end{proof}
\begin{example}\label{ex:B3}
	Let $\mathcal{B}_3$ be 
	$\mathbb{P}_{\mathbb{CP}^1}(\mathcal{O}\oplus \mathcal{O}\oplus \mathcal{O}(1))$
	which is the blow up of $\mathbb{CP}^3$ along $\mathbb{CP}^1$.
	The support of the KRF-MIS on $\mathcal{B}_3$ of exponent $\gamma$ is
	the exceptional divisor of the blow up
	for $\gamma \in (\frac{1}{3}, 1)$.
\end{example}
\begin{proof}
The vertices of the Fano polytope of $\mathcal{B}_3$
is
\[
	({}^tq^{(1)}, {}^tq^{(2)}, {}^tq^{(3)}, {}^tq^{(4)}, {}^tq^{(5)})
	=
	\left(\begin{array}{ccccc}
		  1 &  0 &  1 & -1 &  0 \\
		  1 &  0 &  1 &  0 & -1 \\
		  1 & -1 &  0 &  0 &  0
	\end{array}\right).
\]
This toric Fano manifold has a symmetry which permutes $\{q^{(1)}, q^{(2)}\}$ and permutes $\{q^{(4)}, q^{(5)}\}$, then it is contained in $\mathcal{W}_1$ and $N^{\mathcal{W}(\mathcal{B}_3)}_\mathbb{R}$ is generated by a vector $(1, 1, 0)$.
The vertices of the polytope $P^*$ is
\[
	\left(\begin{array}{cccccc}
		-1 &  2 & -1 &  2 & -1 & -1\\
		-1 & -1 &  2 & -1 &  2 & -1\\
		-1 & -1 & -1 &  0 &  0 &  3
	\end{array}\right).
\]
From (\ref{eq:criterion_KRS}), we find that $\beta_{KRS}=\beta(1,1,0)$, where $\beta>0$.
The vertex of $P$ contained in $int(\widetilde{st(x(-\beta_{KRS})}))$ is
$\{q^{(3)}\}$, 
which represents the exceptional divisor.
Then, Corollary \ref{cor:star_cor} implies that the support of the KRF-MIS of exponent $\gamma$ is 
the exceptional divisor
where $\gamma$ is strictly smaller than $1$ and sufficiently close to $1$.
Its complex singularity exponent is
$\frac{1}{3}$.
Therefore the proof is completed.
\end{proof}
\begin{example}\label{ex:C1}
	Let $\mathcal{C}_1$ be 
	$\mathbb{P}_{\mathbb{CP}^1\times \mathbb{CP}^1}(\mathcal{O}\oplus \mathcal{O}(1,1))$.
	The support of the KRF-MIS on $\mathcal{C}_1$ of exponent $\gamma$ is
	$S_\infty$
	for $\gamma \in (\frac{1}{2}, 1)$.
	Here $S_\infty$ is the divisor defined by a section 
	$(0, \sigma_1\otimes \sigma_2)$ of
	$\mathcal{O}\oplus \mathcal{O}(1,1)$ over $\mathbb{CP}^1\times \mathbb{CP}^1$ and
	$\sigma_i$ is the pull-back of the section of $\mathcal{O}_{\mathbb{CP}^1}(1)$
	with respect to the $i$-th projection $\mathbb{CP}^1\times \mathbb{CP}^1 \to \mathbb{CP}^1$. Remark that $S_\infty$ is not an exceptional divisor.
\end{example}
\begin{proof}
The vertices of the Fano polytope of $X$
is
\[
	({}^tq^{(1)}, {}^tq^{(2)}, {}^tq^{(3)}, {}^tq^{(4)}, {}^tq^{(5)}, {}^tq^{(6)})
	=
	\left(\begin{array}{cccccc}
		  0 &  0 &  1 &  0 & -1 &  0\\
		  0 &  0 &  0 &  1 &  0 & -1\\
		  1 & -1 &  1 &  1 &  0 &  0
	\end{array}\right).
\]
This toric Fano manifold has a symmetry which permutes $\{q^{(3)}, q^{(4)}, q^{(5)}, q^{(6)}\}$, then it is contained in $\mathcal{W}_1$ and $N^{\mathcal{W}(\mathcal{C}_1)}_\mathbb{R}$ is generated by a vector $(0, 0, 1)$.
The vertices of the polytope $P^*$ is
\[
	\left(\begin{array}{cccccccc}
		 0 &  0 & -1 & -1 &  2 &  2 & -1 & -1\\
		 0 & -1 &  0 & -1 & -1 &  2 &  2 & -1\\
		 1 &  1 &  1 &  1 & -1 & -1 &  -1 & -1
	\end{array}\right).
\]
From (\ref{eq:criterion_KRS}), we find that $\beta_{KRS}=\beta(0,0,1)$, where $\beta>0$.
The vertex of $P$ contained in $int(\widetilde{st(x(-\beta_{KRS})}))$ is
$\{q^{(6)}\}$, 
which represents 
$S_\infty$.
Then, Corollary \ref{cor:star_cor} implies that the support of the KRF-MIS of exponent $\gamma$ is 
$S_\infty$
where $\gamma$ is strictly smaller than $1$ and sufficiently close to $1$.
Its complex singularity exponent is
$\frac{1}{2}$.
Therefore the proof is completed.
\end{proof}
\begin{example}\label{ex:C4}
	Let $\mathcal{C}_4$ be 
	$(\mathbb{CP}^2 \sharp \overline{\mathbb{CP}^2})\times \mathbb{CP}^1$,
	which is the blow up of $\mathbb{CP}^2\times \mathbb{CP}^1$ along
	$\{\mbox{point}\}\times \mathbb{CP}^1$.
	The support of the KRF-MIS on $\mathcal{C}_4$ of exponent $\gamma$ is
	the exceptional divisor of the blow up
	for $\gamma \in (\frac{1}{2}, 1)$.
\end{example}
\begin{proof}
The vertices of the Fano polytope of $\mathcal{C}_4$
is
\[
	({}^tq^{(1)}, {}^tq^{(2)}, {}^tq^{(3)}, {}^tq^{(4)}, {}^tq^{(5)}, {}^tq^{(6)})
	=
	\left(\begin{array}{cccccc}
		  0 &  0 & -1 & -1 &  1 &  0\\
		  0 &  0 & -1 &  0 &  1 & -1\\
		  1 & -1 &  0 &  0 &  0 &  0
	\end{array}\right).
\]
This toric Fano manifold has a symmetry which permutes $\{q^{(1)}, q^{(2)}\}$ and permutes $\{ q^{(4)}, q^{(6)}\}$, then it is contained in $\mathcal{W}_1$ and $N^{\mathcal{W}(\mathcal{C}_4)}_\mathbb{R}$ is generated by a vector $(-1, -1, 0)$.
The vertices of the polytope $P^*$ is
\[
	\left(\begin{array}{cccccccc}
		 0 & -1 &  2 & -1 &  0 & -1 &  2 & -1\\
		-1 &  2 & -1 &  2 & -1 &  2 & -1 &  2\\
		 1 &  1 &  1 &  1 & -1 & -1&  -1 & -1
	\end{array}\right).
\]
From (\ref{eq:criterion_KRS}), we find that $\beta_{KRS}=\beta(-1,-1,0)$, where $\beta>0$.
The vertex of $P$ contained in $int(\widetilde{st(x(-\beta_{KRS})}))$ is
$\{q^{(3)}\}$, 
which represents 
the exceptional divisor.
Then, Corollary \ref{cor:star_cor} implies that the support of the KRF-MIS of exponent $\gamma$ is 
$S_\infty$
where $\gamma$ is strictly smaller than $1$ and sufficiently close to $1$.
Its complex singularity exponent is
$\frac{1}{2}$.
Therefore the proof is completed.
\end{proof}
Next we consider a $(\mathbb{CP}^2 \sharp 2\overline{\mathbb{CP}^2})$-bundle $\mathcal{E}_1$ over $\mathbb{CP}^1$.
This manifold is derived as follows.
Let $\tilde{E}_0$ be its exceptional divisor of the blow up $\pi: \mathcal{B}_3 \to \mathbb{CP}^3$ along a curve
\[
	F_0:=\{[z_0; z_1: 0: 0] \in \mathbb{CP}^3 \mid z_i \in \mathbb{C}\}
	\simeq
	\mathbb{CP}^1.
\] 
Let $\tilde{F}_1$ and $\tilde{F}_2$ are the two ($T_{\mathbb{C}}$-fixed) curves which are reduced to $F_0$ under $\pi$.
Then $\mathcal{E}_1$ is constructed from the blow up of $\mathcal{B}_3$ along $\tilde{F}_1$ and $\tilde{F}_2$.
Let $\tilde{\tilde{E}}_0$ be the proper transform of $\tilde{E}_0$ and $\cup_{i=1,2}\tilde{\tilde{E}}_i$  be the exceptional divisors with respect to the blow up of $\mathcal{B}_3$. 
Remark that $\tilde{\tilde{E}}_0$ is not exceptional in $\mathcal{E}_1$.
\begin{example}\label{ex:E1}
	Let $\mathcal{E}_1$ be a
	$(\mathbb{CP}^2 \sharp 2\overline{\mathbb{CP}^2})$-bundle 
	 over $\mathbb{CP}^1$ defined as above.
	The support of the KRF-MIS on $\mathcal{E}_1$ of exponent $\gamma$ is
	\[
		\left\{\begin{array}{cc}
		\tilde{\tilde{E}}_0 \cup (\cup_{i=1,2}\tilde{\tilde{E}}_i) & \mbox{for } \gamma \in (\frac{1}{2}, 1)  \\
		\tilde{\tilde{E}}_0 & \mbox{for } \gamma \in (\frac{1}{3}, \frac{1}{2}).
		\end{array}\right.
	\]
\end{example}
\begin{proof}
The vertices of the Fano polytope of $\mathcal{E}_1$
is
\[
	({}^tq^{(1)}, {}^tq^{(2)}, {}^tq^{(3)}, {}^tq^{(4)}, {}^tq^{(5)}, {}^tq^{(6)}, {}^tq^{(7)})
	=
	\left(\begin{array}{ccccccc}
		  1 &  0 & -1 &  0 &  1 &  1 &  0\\
		  1 &  1 &  0 & -1 &  0 &  1 &  0\\
		  0 &  0 &  0 &  0 &  0 &  1 & -1 
	\end{array}\right)
\]
This toric Fano manifold has a symmetry which permutes $\{q^{(2)}, q^{(5)}\}$ and permutes $\{ q^{(6)}, q^{(7)}\}$, then it is contained in $\mathcal{W}_1$ and $N^{\mathcal{W}(\mathcal{E}_1)}_\mathbb{R}$ is generated by a vector $(1, 1, 0)$.
The vertices of the polytope $P^*$ is
\[
	\left(\begin{array}{cccccccccc}
		 0 & -1 & -1 &  1 &  1 &  0 & -1 & -1 &  1 &  1\\
		 1 &  1 & -1 & -1 &  0 &  1 &  1 & -1 & -1 &  0 \\
		 0 &  1 &  3 &  1 &   0 & -1& -1 & -1 & -1 & -1
	\end{array}\right).
\]
From (\ref{eq:criterion_KRS}), we find that $\beta_{KRS}=\beta(1,1,0)$, where $\beta>0$.
The vertices of $P$ contained in $int(\widetilde{st(x(-\beta_{KRS})}))$ are
$\{q^{(1)}, q^{(2)}, q^{(5)}\}$. 
Remark that $\{q^{(1)}\}$ represents $\tilde{\tilde{E}}_0$ and $\{q^{(2)}, q^{(5)}\}$ represents $\{\tilde{\tilde{E}}_1, \tilde{\tilde{E}}_2\}$.
Then, Corollary \ref{cor:star_cor} implies that the support of the KRF-MIS of exponent $\gamma$ is 
the sum of $\tilde{\tilde{E}}_0$, $\tilde{\tilde{E}}_1$ and $\tilde{\tilde{E}}_2$
where $\gamma$ is strictly smaller than $1$ and sufficiently close to $1$.
Their complex singularity exponents are
\[
	\left\{\begin{array}{cc}
		\frac{1}{2} & \mbox{for }  \tilde{\tilde{E}}_1, \,\,\, \tilde{\tilde{E}}_2\\
		\frac{1}{3} & \mbox{for } \tilde{\tilde{E}}_0.
	\end{array}\right.
\]
Therefore the proof is completed.
\end{proof}
\begin{example}\label{ex:E3}
	Let $\mathcal{E}_3$ be 
	$(\mathbb{CP}^2 \sharp 2\overline{\mathbb{CP}^2})\times \mathbb{CP}^1$,
	which is the blow up of $\mathbb{CP}^1 \times \mathbb{CP}^1 \times
	\mathbb{CP}^1$ along $\{p_1\}\times \{p_2\}\times \mathbb{CP}^1$. 
	The support of the KRF-MIS on $\mathcal{E}_3$ of exponent $\gamma$ is
	\[
		\left\{\begin{array}{cc}
		\cup_{i=0}^{2} E_i & \mbox{for } \gamma \in (\frac{1}{2}, 1)  \\
		E_0 & \mbox{for } \gamma \in (\frac{1}{3}, \frac{1}{2}).
		\end{array}\right.
	\]
	Here $E_0$ denotes the exceptional divisor of the blow up and
	$E_1$ (resp. $E_2$) denotes the proper transform of 
	$\mathbb{CP}^1\times \{p_2\} \times \mathbb{CP}^1$
	(resp. $ \{p_1\} \times \mathbb{CP}^1 \times \mathbb{CP}^1$) 
\end{example}
\begin{proof}
The vertices of the Fano polytope of $\mathcal{E}_3$
is
\[
	({}^tq^{(1)}, {}^tq^{(2)}, {}^tq^{(3)}, {}^tq^{(4)}, {}^tq^{(5)}, {}^tq^{(6)}, {}^tq^{(7)})
	=
	\left(\begin{array}{ccccccc}
		  1 &  0 & -1 &  0 &  1 &  0 &  0\\
		  1 &  1 &  0 & -1 &  0 &  0 &  0\\
		  0 &  0 &  0 &  0 &  0 &  1 & -1 
	\end{array}\right).
\]
This toric Fano manifold has a symmetry which permutes $\{q^{(2)}, q^{(5)}\}$ and permutes $\{ q^{(6)}, q^{(7)}\}$, then it is contained in $\mathcal{W}_1$ and $N^{\mathcal{W}(\mathcal{E}_3)}_\mathbb{R}$ is generated by a vector $(1, 1, 0)$.
The vertices of the polytope $P^*$ is
\[
	\left(\begin{array}{cccccccccc}
		 0 & -1 & -1 &  1 &  1 &  0 & -1 & -1 &  1 &  1\\
		 1 &  1 & -1 & -1 &  0 &  1 &  1 & -1 & -1 &  0 \\
		 1 &  1 &  1 &  1 &   1 & -1& -1 & -1 & -1 & -1
	\end{array}\right).
\]
From (\ref{eq:criterion_KRS}), we find that $\beta_{KRS}=\beta(1,1,0)$, where $\beta>0$.
The vertices of $P$ contained in $int(\widetilde{st(x(-\beta_{KRS})}))$ are
$\{q^{(1)}, q^{(2)}, q^{(5)}\}$. 
Remark that $\{q^{(1)}\}$ represents $E_0$ and $\{q^{(2)}, q^{(5)}\}$ represents $\{E_1, E_2\}$.
Then, Corollary \ref{cor:star_cor} implies that the support of the KRF-MIS of exponent $\gamma$ is 
the sum of $E_0$, $E_1$ and $E_2$
where $\gamma$ is strictly smaller than $1$ and sufficiently close to $1$.
Their complex singularity exponents are
\[
	\left\{\begin{array}{cc}
		\frac{1}{2} & \mbox{for }  E_1, \,\,\, E_2\\
		\frac{1}{3} & \mbox{for } E_0
	\end{array}\right.
\]
Therefore the proof is completed.
\end{proof}
Finally let us consider a $(\mathbb{CP}^2 \sharp 3\overline{\mathbb{CP}^2})$-bundle  $\mathcal{F}_2$ over $\mathbb{CP}^1$.
Let $\tilde{E}_0$, $\tilde{\tilde{E}}_i$ $(i=0,1,2)$, $F_0$, and $\tilde{F}_i$ $(i=1, 2)$ be as in Example \ref{ex:E1}.
Let $\tilde{\pi}: \mathcal{E}_1 \to \mathcal{B}_3$ be the blow up of $\mathcal{B}_3$ along $\tilde{F}_1$ and $\tilde{F}_2$.
Let $F_3$ be the $\mathbb{CP}^1$ in $\mathbb{CP}^3$ defined by
\[
	F_3:=\{[0: 0: z_3; z_4] \mid z_i \in \mathbb{C}\}.
\]
The manifold $\mathcal{F}_2$ is constructed from the blow up of $\mathcal{E}_1$ along the curve $\tilde{\pi}^{-1}(\pi^{-1}(F_3))$.
Let $\tilde{\tilde{\tilde{E}}}_0$ be the proper transform of $\tilde{\tilde{E}}_0$ with respect to the blow up of $\mathcal{E}_1$ along the curve $\tilde{\pi}^{-1}(\pi^{-1}(F_3))$.
Remark that $\tilde{\tilde{\tilde{E}}}_0$ is not an exceptional divisor.
\begin{example}\label{ex:F2}
	Let $\mathcal{F}_2$ be a
	$(\mathbb{CP}^2 \sharp 3\overline{\mathbb{CP}^2})$-bundle 
	 over $\mathbb{CP}^1$ defined as above.
	The support of the KRF-MIS on $\mathcal{F}_2$ of exponent $\gamma$ is
	$\tilde{\tilde{\tilde{E}}}_0$
	for $\gamma \in (\frac{1}{2}, 1)$.
\end{example}
\begin{proof}
The vertices of the Fano polytope of $\mathcal{F}_2$
is
\[
	({}^tq^{(1)}, {}^tq^{(2)}, {}^tq^{(3)}, {}^tq^{(4)}, 
	{}^tq^{(5)}, {}^tq^{(6)}, {}^tq^{(7)}, {}^tq^{(8)})
	=
	\left(\begin{array}{cccccccc}
		  1 &  0 & -1 & -1 &  0 &  1 &  1 &  0\\
		  0 &  1 &  1 &  0 & -1 & -1 &  0 &  0\\
		  0 &  0 &  0 &  0 &  0 &  0 &  1 & -1
	\end{array}\right).
\]
This toric Fano manifold has a symmetry which permutes $\{q^{(2)}, q^{(6)}\}$ and permutes $\{ q^{(7)}, q^{(8)}\}$, then it is contained in $\mathcal{W}_1$ and $N^{\mathcal{W}(\mathcal{F}_2)}_\mathbb{R}$ is generated by a vector $(1, 0, 0)$.
The vertices of the polytope $P^*$ is
\[
	\left(\begin{array}{cccccccccccc}
		 1 &  0 & -1 & -1 &  0 &  1 &  1 &  0 & -1 & -1 &  0 &  1\\
		 1 &  1 &  0 & -1 & -1 &  0 &  1 &  1 &  0 & -1 & -1 &  0\\
		 0 &  1 &  2 &  2 &  1 &  0 & -1 & -1 & -1 & -1 & -1 & -1
	\end{array}\right).
\]
From (\ref{eq:criterion_KRS}), we find that $\beta_{KRS}=\beta(1, 0,0)$, where $\beta>0$.
The vertex of $P$ contained in $int(\widetilde{st(x(-\beta_{KRS})}))$ is
$\{q^{(1)}\}$, 
which represents 
$\tilde{\tilde{\tilde{E}}}_0$.
Then, Corollary \ref{cor:star_cor} implies that the support of the KRF-MIS of exponent $\gamma$ is 
where $\gamma$ is strictly smaller than $1$ and sufficiently close to $1$.
Its complex singularity exponent is
$\frac{1}{2}$.
Therefore the proof is completed.
\end{proof}

By the similar calculation as Example \ref{ex:B3}, we find
\begin{example}\label{ex:blw_CPk}
	Let $X_k$ be the blow up of $\mathbb{CP}^n$ along $\mathbb{CP}^k$, where $0 \le k \le n-2$.
	The support of the KRF-MIS of complex singular exponent $\gamma$ is
	the exceptional divisor 
	for $\gamma \in (\frac{1}{k+2}, 1)$.
\end{example}

\end{document}